\documentclass[11pt]{article}

\usepackage[english]{babel}
\usepackage{amsmath,amsfonts,amssymb,amsthm}
\usepackage{latexsym}
\usepackage{makeidx}
\usepackage{color}
\usepackage{vmargin}
\usepackage{enumerate}
\setmarginsrb{20mm}{20mm}{20mm}{20mm}%
            {5mm}{5mm}{5mm}{10mm}

\numberwithin{equation}{section}

\numberwithin{table}{section}

\usepackage{accents}
\newlength{\dhatheight}
\newcommand{\doublehat}[1]{%
    \settoheight{\dhatheight}{\ensuremath{\widehat{#1}}}%
    \addtolength{\dhatheight}{-0.35ex}%
    \widehat{\vphantom{\rule{0pt}{\dhatheight}}%
    \smash{\widehat{#1}}}}

\usepackage{caption}

\usepackage[all]{xy}

\newtheorem{definition}{Definition}[section]
\newtheorem{lemma}[definition]{Lemma}
\newtheorem{theorem}[definition]{Theorem}
\newtheorem{corollary}[definition]{Corollary}

\newtheorem{proposition}[definition]{Proposition}

\newtheorem{question}[definition]{Question}

\theoremstyle{remark}
\newtheorem{example}[definition]{Example}
\newtheorem{remark}[definition]{Remark}

\newcommand{\bigslant}[2]{{\raisebox{.3em}{$#1$}\left/\raisebox{-.3em}{$#2$}\right.}}

\newcommand{\FF}{\mathbb F}
\newcommand{\K}{\mathbb K}
\newcommand{\N}{\mathbb N}
\newcommand{\Z}{\mathbb Z}

\newcommand{\B}{\mathcal{B}}
\newcommand{\BV}{\B(V)}

\newcommand{\LLC}{_{\K}\mathrm{LLC}}
\newcommand{\LC}{_{\K}\mathrm{LC}}

\newcommand{\caL}{\mathcal L}

\def\f{\phi}
\def\ent{\mathrm{ent}}
\def\Vect{\mathrm{Vect}}

\def\End{\mathrm{End}}

\def\Aut{\mathrm{Aut}}

\def\End{\mathrm{End}}
\def\Aut{\mathrm{Aut}}
\def\Im{\mathrm{Im}}

\global\def\card#1{\left|{#1}\right|}
\global\def\ca#1{\mathcal{#1}}

\title{Topological entropy for locally linearly compact vector spaces\footnote{This work was supported by Programma SIR 2014 by MIUR  (Project GADYGR, Number RBSI14V2LI, cup G22I15000160008), 
and partially also by the ``National Group for Algebraic and Geometric Structures, and their Applications'' (GNSAGA - INdAM).}}

\author{Ilaria Castellano\footnote{The first author was partially supported by  EPSRC Grant N007328/1 Soluble Groups and Cohomology.}\\{\footnotesize {\tt ilaria.castellano88@gmail.com}} \\ {\footnotesize University of Southampton}\\ {\footnotesize Building 54, Salisbury Road - 50171BJ Southampton (UK)} 
\and Anna Giordano Bruno\footnote{Corresponding author}\\{\footnotesize {\tt anna.giordanobruno@uniud.it}}\\{\footnotesize Universit\`a degli Studi di Udine}\\{\footnotesize Via delle Scienze 206 - 33100 Udine (Italy)} }

\date{}

\begin{document}

\maketitle

\abstract{In analogy to the topological entropy for continuous endomorphisms of totally disconnected locally compact groups, we introduce a notion of topological entropy for continuous endomorphisms of locally linearly compact vector spaces. We study the fundamental properties of this entropy and we prove the Addition Theorem, showing that the topological entropy is additive with respect to short exact sequences. By means of Lefschetz Duality, we connect the topological entropy to the algebraic entropy in a Bridge Theorem.}



\section{Introduction}\label{intro}

In \cite{AKM} Adler, Konheim and McAndrew introduced a notion of topological entropy for continuous self-maps of compact spaces. Later on, in \cite{B}, Bowen gave a definition of topological entropy for uniformly continuous self-maps of metric spaces, that was extended by Hood in \cite{H} to uniform spaces. This notion of entropy coincides with the one for compact spaces (when the compact topological space is endowed with the unique uniformity compatible with the topology), and it can be computed for any given continuous endomorphism $\phi:G\to G$ of a topological group $G$ (since $\phi$ turns out to be uniformly continuous with respect to the left uniformity of $G$). 
In particular, if $G$ is a totally disconnected locally compact group, by van Dantzig's Theorem (see~\cite{vD}) the family $\B_{gr}(G)=\{U\leq G\mid U \text{compact open}\}$ is a neighborhood basis at $0$ in $G$, and the topological entropy of the continuous endomorphism $\phi:G\to G$ can be computed as follows (see \cite{DSV,GBVirili}).
For a subset $F$ of $G$ and for every $n\in\N_+$, the \emph{$n$-th $\phi$-cotrajectory} of $F$ is $$C_n(\phi,F)=F\cap\phi^{-1}F\cap\ldots\cap\phi^{-n+1}F.$$ 
 the \emph{topological entropy of $\f$ with respect to $U\in\B_{gr}(G)$} is
$$H_{top}(\f,U)=\lim_{n\to\infty}\frac{1}{n}\log[U:C_n(\f,U)],$$
and the \emph{topological entropy} of $\phi$ is
$$h_{top}(\f)=\sup\{ H_{top}(\f,U)\mid U\in\B_{gr}(G)\}.$$

A fundamental property of the topological entropy is the so-called Addition Theorem: it holds for a topological group $G$, a continuous endomorphism $\phi:G\to G$ and a closed $\phi$-invariant (i.e., $\phi H\leq H$) normal subgroup of $G$, if
\begin{equation*}\label{AT-eq}
h_{top}(\phi)=h_{top}(\phi\restriction_H)+h_{top}(\overline\phi),
\end{equation*}
where $\overline\phi:G/H\to G/H$ is the continuous endomorphism induced by $\phi$.

The Addition Theorem for continuous endomorphisms of compact groups was deduced in \cite[Theorem 8.3]{Dik+Manolo} from the metric case proved in a more general setting in \cite[Theorem 19]{B}; the separable case was settled by Yuzvinski in \cite{Y}. Recently, in \cite{GBVirili}, the Addition Theorem was proved for topological automorphisms of totally disconnected locally compact groups; more precisely, taken $G$ a totally disconnected locally compact group, $\phi:G\to G$ a continuous endomorphism and $H$ a closed $\phi$-invariant subgroup of $G$,  if $\phi\restriction_H$ is surjective and the continuous endomorphism $\overline\phi:G/H\to G/H$ induced by $\phi$ is injective, then $$h_{top}(\phi)=h_{top}(\phi\restriction_H)+h_{top}(\overline\phi).$$
%
%
The validity of the Addition Theorem in full generality for continuous endomorphisms of locally compact groups remains an open problem, even in the totally disconnected (abelian) case.

\medskip
In this paper we introduce a notion of topological entropy $\ent^*$ for locally linearly compact vector spaces in analogy to the topological entropy $h_{top}$ for totally disconnected locally compact groups. We recall that a topological vector space $V$ over a discrete field $\K$ is a locally linearly compact vector space if it admits a neighborhood basis at $0$ of linearly compact open linear subspaces (see~\cite{Lef}). Denote by $\BV$ be the neighborhood basis at $0$ in $V$ consisting of all linearly compact open linear subspaces of $V$. 

\begin{definition}
Let $V$ be a locally linearly compact vector space and $\phi:V\to V$ a continuous endomorphism.
The \emph{topological entropy of $\f$ with respect to $U\in\BV$} is
\begin{equation}\label{eq:H*}
H^*(\f,U)=\lim_{n\to\infty}\frac{1}{n}\dim \frac{U}{C_n(\f,U)},
\end{equation}
and the \emph{topological  entropy} of $\f\colon V\to V$ is
\begin{equation*}\label{def:ent*}
\ent^*(\f)=\sup\{H^*(\f,U)\mid U\in\BV\}.
\end{equation*}
\end{definition}

The limit in \eqref{eq:H*} exists (see Proposition~\ref{prop:exist}) and $\ent^*$ is always zero on discrete vector spaces (see Corollary~\ref{cor:ent*discrete}); moreover, this entropy admits all the fundamental properties expected from an entropy function (see \S~\ref{fp-sec}).

\smallskip
One of the main results of the present paper is the Addition Theorem for locally linearly compact vector spaces and their continuous endomorphisms:

\begin{theorem}[Addition Theorem]\label{AT-intro}
Let $V$ be a locally linearly compact vector space, $\phi:V\to V$ a continuous endomorphism, $W$ a closed $\phi$-invariant linear subspace of $V$ and $\overline \phi:V/W\to V/W$ the continuous endomorphism induced by $\phi$. Then
$$\ent^*(\phi)=\ent^*(\phi\restriction_W)+\ent^*(\overline\phi).$$
\end{theorem}

In case $V$ is a locally linearly compact vector space over a discrete finite field $\mathbb F$, then $V$ is a totally disconnected locally compact abelian group (see Proposition~\ref{prop:tdlc}(b)) and $$h_{top}(\phi)=\ent^*(\phi)\cdot\log|\mathbb F|$$
(see Proposition~\ref{lem:tdlc}). So, with respect to the general problem of the validity of the Addition Theorem for the topological entropy $h_{top}$ of continuous endomorphisms of locally compact groups, Theorem~\ref{AT-intro} covers the case of those totally disconnected locally compact abelian groups that are also locally linearly compact vector spaces.

To prove the Addition Theorem, we restrict first to the case of continuous endomorphisms of linearly compact vector spaces (see \S~\ref{s:rlc}), and then to topological automorphisms (see \S~\ref{s:redauto}). The technique used for the latter reduction was suggested us by Simone Virili; in fact, Virili and Salce used it in another context in \cite{SV} giving credit to Gabriel \cite{Gabriel}.
Finally, in Section~\ref{s:AT}, we prove the Addition Theorem for topological automorphisms (see Proposition~\ref{thm:ATauto}), so that we can deduce the Addition Theorem for all continuous endomorphisms of linearly compact vector spaces (see Proposition~\ref{prop:ATlc}) and then the Addition Theorem in full generality.

\smallskip
A fundamental tool in the proof of the Addition Theorem is the so-called Limit-free Formula (see Proposition~\ref{prop:lff}), that permits to compute the topological entropy avoiding the limit in the definition in Equation~\eqref{eq:H*}. Indeed, taken $V$ a locally linearly compact vector space and $\phi:V\to V$ a continuous endomorphism, for every $U\in\BV$ we construct a linearly compact linear subspace $U_+$ of $V$ (see Definition~\ref{def:uplus}) such that $U_+$ is an open linear subspace of $\phi U_+$ of finite codimension and
\begin{equation*}
H^*(\phi,U)=\dim \frac{\phi U_+}{U_+}.
\end{equation*}
This result is the counterpart of the same formula for the topological entropy $h_{top}$ of continuous endomorphisms of totally disconnected locally compact groups given in \cite[Proposition 3.9]{GBVirili} (see also \cite{DGB-lff} for the compact case and \cite{GB} for the case of topological automorphisms).
Note that a first Limit-free Formula was sketched by Yuzvinski in \cite{Y} in the context of the algebraic entropy for discrete abelian groups; this was later proved in a slightly more general setting in \cite{DGB-lff} (and extended in \cite[Lemma 5.4]{yuzapp} to a Limit-free Formula for the intrinsic algebraic entropy of automorphisms of abelian groups).

\medskip
In \cite{AKM} Adler, Konheim and McAndrew also sketched a definition of algebraic entropy for endomorphisms of abelian groups, that was later reconsidered by Weiss in \cite{Weiss}, and recently by Dikranjan, Goldsmith, Salce and Zanardo for torsion abelian groups in \cite{Aeag}.
Later on, using the definitions of algebraic entropy given by Peters in \cite{Peters1,Peters2}, the algebraic entropy $h_{alg}$ was extended in several steps (see \cite{DGBpet,Virili}) to continuous endomorphisms of locally compact abelian groups.

In \cite{Weiss} Weiss connected, in a so-called Bridge Theorem, the topological entropy $h_{top}$ of a continuous endomorphism $\phi:G\to G$ of a totally disconnected compact abelian group $G$ to the algebraic entropy $h_{alg}$ of the dual endomorphism $\phi{^\wedge}:G^\wedge\to G^\wedge$ of the Pontryagin dual $G^\wedge$ of $G$, by showing that 
$$h_{top}(\phi)=h_{alg}(\phi^\wedge).$$
The same connection was given by Peters in \cite{Peters1} for topological automorphisms of metrizable compact abelian groups; moreover, these results were recently extended to continuous endomorphisms of compact abelian groups in \cite{DGB-BT}, to continuous endomorphisms of totally disconnected locally compact abelian groups in \cite{GBD}, and to topological automorphisms of locally compact abelian groups in \cite{Virili-BT} (in a much more general setting).
The problem of the validity of the Bridge Theorem in the general case of continuous endomorphisms of locally compact abelian groups is still open.

\smallskip
In \cite{GBSalce} the dimension entropy $\ent_{\dim}$ was studied for endomorphisms of discrete vector spaces, as a particular interesting case of the $i$-entropy $\ent_i$ for endomorphisms of modules over a ring $R$ and an invariant $i$ of $\mathrm{Mod}(R)$ introduced in \cite{SZ} as another generalization of Weiss' entropy.
The dimension entropy is extended in \cite{CGB} to locally linearly compact vector spaces, as follows. 
Let $\f\colon V\to V$ be a continuous endomorphism of a locally linearly compact vector space $V$. For every $U\in\BV$ and $n\in\N_+$, the \emph{$n$-th $\phi$-trajectory of $U$} is $$T_n(\f,U)=U+\f U+\ldots+\f^{n-1}U.$$
The \emph{algebraic entropy of $\f$ with respect to $U$} is
$$H(\f,U)=\lim_{n\to\infty} \frac{1}{n}\dim\frac{T_n(\f,U)}{U},$$
and the \emph{algebraic entropy of $\f$} is
$$\ent(\f)=\sup\{H(\f,U)\mid U\in\BV\}.$$

The second main result of this paper is the following Bridge Theorem, proved in Section~\ref{s:bridge}, connecting the topological entropy $\ent^*$ with the algebraic entropy $\ent$ from \cite{CGB} by means of Lefschetz Duality (see \S~\ref{ss:lefdual}). For a locally linearly compact vector space $V$ and a continuous endomorphism $\phi:V\to V$, we denote by $\widehat V$ the dual of $V$ and by $\widehat\phi:\widehat V\to \widehat V$ the dual endomorphism of $\phi$ with respect to Lefschetz Duality.

\begin{theorem}[Bridge Theorem]\label{BT-intro}
Let $V$ be a locally linearly compact vector space and $\phi:V\to V$ a continuous endomorphism. Then
$$\ent^*(\phi)=\ent(\widehat\phi).$$
\end{theorem}

In analogy to the adjoint algebraic entropy for abelian groups from \cite{DGBS}, the adjoint dimension entropy was considered in \cite{GBSalce}: for a discrete vector space $V$ let $\mathcal C(V)=\{N\leq V\mid \dim V/N<\infty\}$; for an endomorphism $\phi:V\to V$, the \emph{adjoint dimension entropy of $\phi$ with respect to $N$} is $$H^*_{\dim}(\phi,N)=\lim_{n\to\infty}\frac{1}{n}\dim\frac{V}{C_n(\phi,N)},$$ and the \emph{adjoint dimension entropy} is $$\ent_{\dim}^*(\phi)=\sup\{H^*(\phi,N)\mid N\in\mathcal C(V)\}.$$
A Bridge Theorem (see~\cite[Theorem~6.12]{GBSalce}) shows that, for $\phi^*:V^*\to V^*$ the dual endomorphism of the algebraic dual $V^*$ of $V$,
\begin{equation}\label{BT*}
\ent^*_{\dim}(\phi)=\ent_{\dim}(\phi^*).
\end{equation}

Unfortunately, as a consequence of this Bridge Theorem, the adjoint dimension entropy $\ent_{\dim}^*$ is proved to take only the values $0$ and $\infty$ (see \cite[Corollary 6.16]{GBSalce}). So, imitating the same approach used in \cite{GB*} for the adjoint algebraic entropy, a motivating idea to introduce the topological entropy $\ent^*$ in this paper was to ``topologize'' $\ent_{\dim}^*$ so that it admits all possible values in $\N\cup\{\infty\}$. In fact, if $V$ is a linearly compact vector space and $\phi:V\to V$ a continuous endomorphism, then $\BV\subseteq \mathcal C(V)$, furthermore $H^*(\phi,U)=H^*_{\dim}(\phi,U)$ for $U\in\BV$ (see Lemma~\ref{H*lc}), and so $\ent^*(\phi)\leq \ent^*_{\dim}(\phi)$.

Moreover, if $V$ is a discrete vector space and $\phi:V\to V$ is an endomorphism, then $\widehat V$ is linearly compact; it is also known from \cite{CGB} that $\ent(\phi)=\ent_{\dim}(\phi)$, and so Theorem~\ref{BT-intro} gives the equality 
\begin{equation}\label{BTlc}
\ent_{\dim}(\phi)=\ent^*(\widehat\phi),
\end{equation}
that appears to be more natural with respect to that in Equation~\eqref{BT*}.

\medskip
We conclude by leaving an open question about the so-called Uniqueness Theorem. Indeed, a Uniqueness Theorem for the topological entropy in the category of compact groups and continuous homomorphisms was proved by Stojanov in \cite{St}. The same result requires a shorter list of axioms restricting to compact abelian groups (see \cite[Corollary 3.3]{DGB-BT}).


We would say that the Uniqueness Theorem holds for the topological entropy $\ent^*$ in the category of all locally linearly compact vector spaces over a discrete field $\K$ and their continuous homomorphisms, if $\ent^*$ was the unique collection of functions $\ent_V^*:\End(V)\to\N\cup\{\infty\}$, $\phi\mapsto\ent^*(\phi)$,
satisfying for every locally linearly compact vector space $V$ over $\K$: Invariance under conjugation (see Proposition~\ref{prop:basic prop}(a)), Continuity for inverse limits (see Proposition~\ref{prop:basic prop}(e)), Addition Theorem, and $\ent^*({}_F\beta)=\dim F$ for any finite-dimensional vector space $F$ over $\K$, where $V=\bigoplus_{n=-\infty}^0 F\oplus \prod_{n=1}^\infty F$ is endowed with the topology inherited from the product topology of $\prod_{n\in\Z}F$, and ${}_F\beta:V \to V$, $(x_n)_{n\in\Z}\mapsto (x_{n+1})_{n\in\Z}$ is the left Bernoulli shift (see Example~\ref{bernoulli}).

\begin{question}
Does the Uniqueness Theorem hold for the topological entropy $\ent^*$ in the category of locally linearly compact vector spaces over a discrete field $\K$?
\end{question}

The validity of the Uniqueness Theorem for $\ent^*$ in the category of all linearly compact vector spaces over a discrete field $\K$ follows from the Uniqueness Theorem for the dimension entropy $\ent_{\dim}$ in the category of all discrete vector spaces over $\K$ proved in \cite{GBSalce} and Equation~\eqref{BTlc}.

\section{Background on locally linearly compact vector spaces}

\subsection{Locally linearly compact vector spaces}\label{ss:llc}

Fix an arbitrary field $\K$ endowed with the discrete topology. A Hausdorff topological $\K$-vector space $V$ is \emph{linearly topologized} if it admits a neighborhood basis at $0$ consisting of linear subspaces of $V$. Clearly, a linear subspace $W$ of $V$ with the induced topology is still linearly topologized, and the quotient vector space $V/W$ endowed with the quotient topology turns out to be linearly topologized whenever $W$ is also closed in $V$. A finite-dimensional linearly topologized vector space is necessarily discrete.

A \emph{linear variety} of a linearly topologized vector space $V$ is a coset $v+W$, where $W$ is a linear subspace of $V$ and $v\in V$; $v+W$ is \emph{closed} if $W$ is  closed in $V$. 
Following Lefschetz \cite{Lef}, a \emph{linearly compact space}  $V$ is a linearly topologized vector space such that any collection of closed linear varieties of $V$ with the finite intersection property has non-empty intersection. We recall the following known properties that we frequently use in the paper. 


\begin{proposition}[{\cite[page 78]{Lef}, \cite[Propositions 2 and 9]{Mat},\cite[Theorem~28.5]{Warner}}]\label{prop:lc properties}
Let $V$ be a linearly topologized vector space.
\begin{itemize}
\item[(a)] If $W$ is a linearly compact linear subspace of $V$, then $W$ is closed.
\item[(b)] If $V$ is linearly compact and $W$ is a closed linear subspace of $V$, then $W$ is linearly compact;
\item[(c)] If $W$ is another linearly topologized vector space and $f:V\to W$ is continuous homomorphism and $V$ is linearly compact, then $W$ is linearly compact as well.
\item[(d)] If $V$ is discrete, then $V$ is linearly compact if and only if $V$ has finite dimension; hence, if $V$ has finite dimension, then $V$ is linearly compact.
\item[(e)] If $W$ is a closed linear subspace of $V$, then $V$ is linearly compact if and only if $W$ and $V/W$ are linearly compact.
\item[(f)] The direct product of linearly compact vector spaces is linearly compact.
\item[(g)] An inverse limit of linearly compact vector spaces is linearly compact.
\item[(h)] If $V$ is linearly compact, then $V$ is complete.
\end{itemize}
\end{proposition}


The following result ensures that a continuous homomorphism $\f:V\to W$ of linearly topologized vector spaces is also open whenever $V$ is linearly compact.

\begin{proposition}[\protect{\cite[Proposition 1.1(v)]{conn}}] \label{prop:omt}
Let $V$ be a linearly compact space and $W$ a linearly topologized space. If $f\colon V\to W$ is a continuous homomorphism, then $f:V\to fV$ is open. 
In particular, any continuous bijective homomorphism $f\colon V\to W$ is a topological isomorphism.
\end{proposition}

Recall that a \emph{linear filter base} $\ca{N}$ of a linearly compact vector space $V$ is a non-empty family of linear subspaces of $V$ satisfying
$$\forall U,W\in\ca{N},\ \exists Z\in\ca{N},\ \text{such that}\ Z\leq U\cap W.$$

\begin{theorem}[\protect{\cite[Theorem 28.20]{Warner}}]\label{thm:warner}
Let $\ca{N}$ be a linear filter base of a linearly compact vector space $V$.
\begin{enumerate}[(a)]
\item If $W$ is a linearly topologized vector space and $f:V\to W$ a continuous homomorphism, then
$$f\left(\bigcap_{N\in\ca{N}}\overline N\right)=\bigcap_{N\in\ca{N}}\overline{f N}.$$
\item If each member of $\ca{N}$ is closed and if $M$ is a closed linear subspace of $V$, then
$$\bigcap_{N\in\ca{N}}(M+N)=M+\bigcap_{N\in\ca{N}} N.$$
\end{enumerate}
\end{theorem}

\begin{remark}\label{rem:ab5}
In \cite{Gr} Grothendieck introduced axioms for an abelian category $\ca{A}$ concerning the existence and some properties of infinite direct sums and products. In particular, we are interested in the following axiom.

\medskip
\noindent (Ab5*) The category $\ca{A}$ is complete and if $A$ is an object in $\mathcal A$, $\{A_i\}_{i\in I}$ is a lattice of subobjects of $A$ and $B$ is any subobject of $A$, then
$$\bigcap_{i\in I}(B+A_i)=B+\bigcap_{i\in I} A_i.$$
If the abelian category $\ca{A}$ satisfies the axiom (Ab5*), then the inverse limit functor from the category of inverse systems on $\ca{A}$ to $\ca{A}$ is an exact additive functor (see \cite[\S 1.5, \S 1.8]{Gr}).

Now let us consider the complete abelian category $\LC$ of all linearly compact $\K$-vector spaces. The subobjects of a linearly compact vector space $V$ are the closed linear subspaces of $V$. Thus, by Theorem~\ref{thm:warner}(a), the category $\LC$  satisfies the axiom (Ab5*), and so the corresponding inverse limit functor is exact.
\end{remark}

A topological $\K$-vector space $V$ is \emph{locally linearly compact} (briefly, l.l.c.\!) if there exists a neighborhood basis at $0$ in $V$ consisting of linearly compact open linear subspaces of $V$ (see \cite{Lef}). In particular, every l.l.c.\!  vector space $V$ is linearly topologized.  The structure of l.l.c.\! vector spaces is described by the following result.

\begin{theorem}[\protect{\cite[(27.10), page 79]{Lef}}] \label{thm:dec}
A linearly topologized vector space $V$ is l.l.c.\! if and only if $V\cong_{top} V_c\oplus V_d$, where $V_c$ is a linearly compact linear subspace and $V_d$ is a discrete linear subspace of $V$. In particular, $V_c\in\BV$.
\end{theorem}

Thus, every l.l.c.\! vector space is complete. Moreover, the class of all l.l.c.\! vector spaces is closed under taking closed linear subspaces, quotient vector spaces modulo closed linear subspaces and extensions.

\smallskip
We recall that, in view of \cite[Proposition~3]{CGB}, for an l.l.c.\! vector space $V$ and $W$  a closed linear subspace of $V$,
\begin{equation}\label{eq:basis}
\B(W)=\{U\cap W\mid U\in\BV\}\quad\text{and}\quad\B(V/W)=\left\{\frac{U+W}{W}\mid U\in\BV\right\}.
\end{equation}

\subsection{Lefschetz Duality}\label{ss:lefdual}

Let $V$ be an l.l.c.\! vector space and let $\mathrm{CHom(V,\mathbb K)}$ be the vector space of all continuous characters $V\to\K$. 
For a linear subspace $A$ of $V$, the \emph{annihilator} of $A$ in $\mathrm{CHom(V,\mathbb K)}$ is
\begin{equation*}
A^\perp=\{\chi\in\mathrm{CHom(V,\mathbb K)}:\chi(A)=0\}.
\end{equation*}
By \cite[4.(1'), page 86]{Kothe}, the continuous characters in $\mathrm{CHom(V,\mathbb K)}$ separate the points of $V$.

We denote by $\widehat V$ the vector space $\mathrm{CHom(V,K)}$ endowed with the topology having the family
\begin{equation*}
\{A^\perp\mid A\leq V,\ A\ \text{linearly compact}\}
\end{equation*}
as neighborhood basis at $0$. The linearly topologized vector space $\widehat V$ is an l.l.c.\! vector space (see \cite{Lef}). In particular, $\widehat V$ is discrete whenever $V$ is linearly compact since $0=V^\perp$ is open. More generally, $V$ is discrete if and only if $\widehat V$ is linearly compact, and $V$ is linearly compact if and only if $\widehat V$ is discrete. Moreover, if $V$ has finite dimension, then $V$ is discrete and $\widehat V$ is the algebraic dual of $V$, so $\widehat V$ is isomorphic to $V$.

By Lefschetz Duality, $V$ is canonically isomorphic to $\doublehat{V}$; indeed, the canonical map 
\begin{equation}
\omega_V:V\to\doublehat{V}\ \text{such that}\ \omega_V(v)(\chi)=\chi(v)\quad\forall v\in\ V,\forall \chi\in\widehat V,
\end{equation}
 is a topological isomorphism. More precisely, denote by $\LLC$ the category whose objects are all l.l.c.\ vector spaces over $\K$ and whose morphisms are the continuous homomorphisms; let $$\widehat{-}:\LLC\to\LLC$$ be the duality functor, which is defined on the objects by $V\mapsto \widehat V$ and on the morphisms sending $\phi:V\to W$ to $\widehat\phi:\widehat W\to \widehat V$ such that $\phi(\chi)=\chi\circ\phi$ for every $\chi\in\widehat W$. Clearly, the biduality functor $\doublehat{-}\colon\LLC\,\to\, \LLC$ is defined by composing $\widehat{-}$ with itself.

\begin{theorem}[Lefschetz Duality Theorem]\label{thm:ldt}
The biduality functor $\doublehat{-}\colon\LLC\,\to\, \LLC$ and the identity functor $\mathrm{id}\colon\LLC\,\to\,\LLC$ are naturally isomorphic.
\end{theorem}

In particular, the duality functor defines a duality between the subcategory $\LC$ of linearly compact vector spaces over $\K$ and the subcategory ${}_\K\Vect$ of discrete vector spaces over $\K$. 

\smallskip
We recall that, for a continuous homomorphism $f\colon V\to W$ of l.l.c.\! vector spaces.
\begin{enumerate}[(a)]
\item If $f$ is injective and it is open onto its image, then $\widehat f$ is surjective. 
\item If $f$ is surjective, then $\widehat f$ is injective.
\end{enumerate}

As a consequence of Lefschetz Duality Theorem, every linearly compact vector space is a product of one-dimensional vector spaces (see \cite[Theorem 32.1]{Lef}). 
Moreover, one can derive the following result (see \cite[Proposition 4 and Corollary 1]{CGB}); note that every compact linearly topologized vector space is linearly compact.

\begin{proposition}\label{prop:tdlc}
Let $\K$ be a discrete finite field and let $V$ be a $\K$-vector space.
\begin{enumerate}[(a)]
\item If $V$ is linearly compact, then $V$ is compact.
\item If $V$ is l.l.c., then $V$ is locally compact.
\end{enumerate}
\end{proposition}

The above proposition implies that, for an l.l.c.\! vector space $V$ over a discrete finite field, $\BV$ is a neighborhood basis at $0$ in $V$ consisting of compact open subgroups; so, by van Dantzig's Theorem (see \cite{vD}) we obtain the following result.

\begin{corollary}\label{cor:tdlc}
An l.l.c.\! vector space over a discrete finite field is a totally disconnected locally compact abelian group.
\end{corollary}

Given an l.l.c.\! vector space $V$, let $B$ be a linear subspace of $\widehat{V}$. The \emph{annihilator} of $B$ in $V$ is
\begin{equation*}
B^\top=\{x\in V:\chi(x)=0\ \text{for every }\chi\in B\}.
\end{equation*}
For every linear subspace $B$ of $\widehat V$, we have $\omega_V(B^\top)=B^\perp$.

\smallskip
We recall now some known properties of the annihilators (see \cite{Lef,Kothe}) that we use further on.

\begin{lemma}\label{perptop}
Let $V$ be an l.l.c.\! vector space and $A$ a linear subspace of $V$. Then:
\begin{enumerate}[(a)]
\item if $B$ is another linear subspace of $V$ and $A\leq B$, then $B^\perp\leq A^\perp$;
\item $A^\perp=\overline{A}^\perp$;
\item $A^\perp$ is a closed linear subspace of $\widehat{V}$;
\item if $A$ is a closed, $(A^\perp)^\top=A$.
\end{enumerate}
\end{lemma}
%

%

\begin{lemma}\label{lem:intsum}\label{rem:lc ann}
Let $V$ be an l.l.c.\! vector space and $A_1,\ldots,A_n$ linear subspaces of $V$. Then:
\begin{enumerate}[(a)]
\item $\left(\sum_{i=1}^n A_i\right)^\perp=\bigcap_{i=1}^nA_i^\perp$;
\item $\overline{\sum_{i=1}^nA_i^\perp}\subseteq\left(\bigcap_{i=1}^nA_i\right)^\perp;$
\item if $A_1,\ldots,A_n$ are closed, $\left(\bigcap_{i=1}^nA_i\right)^\perp=\overline{\sum_{i=1}^n A_i^\perp}$;
\item if $A_1,\ldots,A_n$ are linearly compact, $\left(\bigcap_{i=1}^n A_i\right)^\perp=\sum_{i=1}^n A_i^\perp$.
\end{enumerate}
\end{lemma}

The following is a well-know result concerning l.l.c.\! vector spaces (see \cite[\S 10.12.(6), \S 12.1.(1)]{Kothe}).

\begin{remark}\label{dualperp}
Let $V$ be an l.l.c.\! vector space and $U$ a closed linear subspace of $V$, then
$$\widehat{V/U}\cong_{top} U^\perp\ \text{and}\ \widehat U\cong_{top} \widehat{V}/U^\perp.$$

The first topological isomorphism is the following.
Let $\pi\colon V\to V/U$ be the canonical projection, consider the continuous injective homomorphism $\widehat\pi\colon\widehat{V/U}\to \widehat V$; noting that $\widehat\pi(\widehat{V/U})=U^\perp$, let
\begin{equation}\label{alpha}
\alpha\colon\widehat{V/U}\to U^\perp,\quad\chi\mapsto\widehat\pi(\chi),
\end{equation}
that turns out to be a topological isomorphism.
%

To find explicitly the second topological isomorphism, let $\iota:U\to V$ be the topological embedding of $U$ in $V$ and consider the continuous surjective homomorphism $\widehat \iota\colon\widehat{V}\to\widehat{U}$; in particular, $\widehat\iota(\chi)=\chi\restriction_U$ for every $\chi\in\widehat{V}$. Since $\ker \widehat\iota=U^\perp$, consider the continuous isomorphism induced by $\widehat\iota$
\begin{equation}\label{beta}
\beta\colon\widehat{V}/U^\perp\to\widehat{U},\quad \chi+U^\perp\mapsto \chi\circ\iota,
\end{equation}
that turns out to be a topological isomorphism.
\end{remark}

A consequence of the above topological isomorphisms is the following relation that we use in the last section.

\begin{lemma}\label{A/B}
Let $V$ be a l.l.c.\! vector space, and let $A,B$ be closed linear subspaces of $V$ such that $B\leq A$. Then $\widehat{A/B}\cong_{top}B^\perp/A^\perp$.
\end{lemma}
\begin{proof}
Let $\iota\colon A/B\to V/B$ be the topological embedding and $\beta\colon \widehat{V/B}/(A/B)^\perp\to\widehat{A/B}$ the topological isomorphism given by Equation~\eqref{beta}. Set $\pi\colon\widehat{V/B}\to  \widehat{V/B}/(A/B)^\perp$ be the canonical projection. Since $\widehat\iota=\beta\circ \pi$, $\widehat\iota$ is an open continuous surjective homomorphism. Let $\alpha:\widehat{V/B}\to B^\perp$ be the topological isomorphism given by Equation~\eqref{alpha}; then $\varphi=\beta\circ\pi\circ\alpha^{-1}$ is an open continuous surjective homomorphism $\varphi:B^\perp\to \widehat{A/B}$. 
\begin{equation*}
\xymatrix{ B^\perp\ar[r]^{\alpha^-1}\ar@/_2pc/[rrr]_{\varphi}&\widehat{V/B}\ar@/^2pc/[rr]^{\widehat\iota} \ar[r]^{\pi}&\frac{\widehat{V/B}}{(A/B)^\perp}\ar[r]^{\beta}& \widehat{A/B}}
\end{equation*}
As $\ker\varphi=A^\perp$, we conclude that $B^\perp/A^\perp\cong_{top}\widehat{A/B}$.
\end{proof}

\section{Properties and examples}

\subsection{Existence of the limit and basic properties}

Let $V$ be an l.l.c.\! vector space, $\phi:V\to V$ a continuous endomorphism and $U\in\BV$. 
For every $n\in\N_+$, $C_n(\phi,U)\in\BV$ by Proposition~\ref{prop:lc properties}(b); hence, $U/C_n(\phi,U)$ has finite dimension by Proposition~\ref{prop:lc properties}(d,e). 

Moreover, $C_n(\phi,U)\geq C_{n+1}(\phi,U)$ for every $n\in\N_+$, so we have the following decreasing chain in $\BV$
$$U=C_1(\phi,U)\geq C_2(\phi,U)\geq\cdots\geq C_{n}(\phi,U)\geq C_{n+1}(\phi,U)\geq\ldots.$$
The largest $\f$-invariant subspace of $U$, namely,
$$C(\phi,U)=\bigcap_{n\in\N_+} C_n(\phi,U),$$
is the \emph{$\phi$-cotrajectory} of $U$ in $V$; it is a linearly compact linear subspace of $V$ by Proposition~\ref{prop:lc properties}(b).

\begin{lemma}\label{cn}
Let $V$ be an l.l.c.\! vector space, $\f:V\to V$ a continuous endomorphism and $U\in\BV$.
For every $n\in\N_+$, $C_{n}(\f,U)/C_{n+1}(\f,U)$ has finite dimension and $C_{n+1}(\f,U)/C_{n+2}(\f,U)$ is isomorphic to a linear subspace of $C_{n}(\f,U)/C_{n+1}(\f,U)$.
\end{lemma}
\begin{proof}
To simplify the notation, let $C_n=C_n(\phi,U)$ for every $n\in\N_+$. 

Fix $n\in\N_+$. Since $C_{n+1}\leq C_n$, and $C_{n+1}$ is open while $C_n$ is linearly compact, $C_n/C_{n+1}$ has finite dimension by Proposition~\ref{prop:lc properties}(d,e). 

Since $C_{n+2}=C_{n+1}\cap\f^{-n-1}U$ and $C_{n+1}=U\cap\f^{-1}C_{n}$, it follows that
\begin{equation*}\label{eq:c1}
\frac{C_{n+1}}{C_{n+2}}\cong \frac{C_{n+1}+\f^{-n-1}U}{\f^{-n-1}U}\leq\frac{\f^{-1}C_{n}+\f^{-n-1}U}{\f^{-n-1}U}.
\end{equation*}
On the other hand, since $C_{n+1}=C_{n}\cap\f^{-n}U$,
\begin{equation*}\label{eq:c2}
\frac{C_{n}}{C_{n+1}}\cong\frac{C_{n}+\f^{-n}U}{\f^{-n}U}.
\end{equation*}
Let $\tilde\f\colon V/\f^{-n-1}U\to V/\f^{-n}U$ be the injective homomorphism induced by $\f$. 
Then $$\tilde\phi\left(\frac{\f^{-1}C_{n}+\f^{-n-1}U}{\f^{-n-1}U}\right)\leq \frac{C_{n}+\f^{-n}U}{\f^{-n}U}$$ and so the thesis follows.
\end{proof}

The following result shows that the limit in the definition of the topological entropy $H^*(\phi,U)$ (see Equation~\eqref{eq:H*}) exists and it is a natural number.

\begin{proposition}\label{prop:exist}
Let $V$ be an l.l.c.\! vector space, $\f:V\to V$ a continuous endomorphism and $U\in\BV$. For every $n\in\N_+$, let
$$\gamma_n=\dim \frac{C_n(\f,U)}{C_{n+1}(\f,U)}.$$
Then the sequence of non-negative integers $\{\gamma_n\}_{n\in\N}$ is decreasing, hence stationary. Moreover,
$H^*(\f,U)=\gamma$, where $\gamma$ is the value of the stationary sequence $\{\gamma_n\}_{n\in\N}$ for $n\in\N_+$ large enough.
\end{proposition}
\begin{proof} 
To simplify the notation, let $C_n=C_n(\phi,U)$ for every $n\in\N_+$. By Lemma~\ref{cn}, $\gamma_{n+1}\leq\gamma_{n}$ for every $n\in\N_+$.
Hence, there exist $\gamma\in\N$ and $n_0\in\N$ such that $\gamma_n=\gamma$ for all $n\geq n_0$. Since
\begin{equation*}
\frac{U}{C_n}\cong\frac{U/C_{n+1}}{C_n/C_{n+1}},
\end{equation*}
it follows that
\begin{equation*}
\dim \frac{U}{C_{n+1}} =\dim \frac{U}{C_{n}} +\gamma_n,
\end{equation*}
and so, for every $n\in\N$,
\begin{equation*}
\dim \frac{U}{C_{n_0+n}} =\dim \frac{U}{C_{n_0}} +n\gamma.
\end{equation*}
Hence, $H^*(\f,U)=\lim_{n\to \infty}\frac{1}{n}\left(\dim \frac{U}{C_{n_0}} +n\gamma\right)=\gamma$.
\end{proof}

We see now that $H^*(\f,-)$ is monotone decreasing in the following sense.

\begin{lemma}\label{lem:mono*}
Let $V$ be an l.l.c.\! vector space, $\f\colon V\to V$ a continuous endomorphism and $U,U'\in\BV$. If $U'\leq U$, then $H^*(\f,U)\leq H^*(\f,U')$.
\end{lemma}
\begin{proof} 
Let $U,U'\in\BV$ such that $U'\leq U$. As $C_n(\f,U')\leq C_n(\f,U)$ for all $n\in\N_+$, from
\begin{equation*}
\frac{U'/C_n(\f,U')}{(U'\cap C_n(\f,U))/C_n(\f,U')}\cong \frac{U'}{U'\cap C_n(\f,U)}\cong\frac{C_n(\f,U)+U'}{C_n(\f,U)},
\end{equation*}
since all terms are finite-dimensional, it follows that 
\begin{equation*}
\dim \frac{U'}{C_n(\f,U')}\geq\dim \frac{C_n(\f,U)+U'}{C_n(\f,U)}
\end{equation*}
Since $C_n(\f,U)\leq C_n(\f,U)+U'\leq U$, 
\begin{eqnarray*}
\dim \frac{C_n(\f,U)+U'}{C_n(\f,U)} &=& \dim \frac{U}{C_n(\f,U)} -\dim\frac{U}{C_n(\f,U)+U'}\\
&\geq& \dim \frac{U}{C_n(\f,U)} - \dim\frac{U}{U'}.
\end{eqnarray*}
Therefore, $\dim \frac{U'}{C_n(\f,U')}\geq\dim \frac{U}{C_n(\f,U)} - \dim\frac{U}{U'}$,
and so $H^*(\f,U')\geq H^*(\f,U)$, since $\dim\frac{U}{U'}$ does not depend on $n$.
\end{proof}

As a straightforward consequence of Lemma~\ref{lem:mono*} it is possible to compute the topological entropy by restricting to any neighborhood basis at $0$ in $V$ contained in $\BV$:

\begin{corollary}\label{cor:coinitial}
Let $V$ be an l.l.c.\! vector space and $\phi:V\to V$ a continuous endomorphism. If $\B\subseteq\B(V)$ is a neighborhood basis at $0$ in $V$, then
$$\ent^*(\f)=\{H^*(\f,U)\mid U\in\B\}.$$
 Consequently, if $W$ is an open linear subspace of $V$, then $\ent^*(\f)=\{H^*(\f,U)\mid U\in\B(W)\}.$
\end{corollary}

We consider now the case of topological entropy zero, starting from a basic example.

\begin{example}\label{idex}
Let $V$ be an l.l.c.\! vector space.
\begin{enumerate}[(a)]
\item If $\phi:V\to V$ is a continuous endomorphism such that $\phi^{-1}U\leq U$ for every $U\in\BV$, then $\ent^*(\phi)=0$.
\item Item (a) implies immediately that $\ent^*(id_V)=0$, where $id_V:V\to V$ is the identity automorphism.
\end{enumerate}
\end{example}

The following result concerns the general case of an endomorphism of zero topological entropy.

\begin{proposition}\label{prop:ent*zero}
Let $V$ be an l.l.c.\! vector space, $\phi:V\to V$ a continuous endomorphism and $U\in\BV$. The following conditions are equivalent:
\begin{itemize}
\item[(a)] $H^*(\phi,U)=0$;
\item[(b)] there exists $n\in\N$ such that $C(\phi,U)=C_n(\phi,U)$;
\item[(c)] $C(\phi,U)$ is open.
\end{itemize}
In particular, $\ent^*(\f)=0$ if and only if $C(\f,U)$ is open for all $U\in\BV$.
\end{proposition}
\begin{proof} 
(a)$\Rightarrow$(b) Since $\dim\frac{C_n(\f,U)}{C_{n+1}(\phi,U)}=0$ eventually by Proposition~\ref{prop:exist}, there exists $n_0\in\N_+$ such that $C_n(\phi,U)=C_{n_0}(\phi,U)$ for every $n\geq n_0$.

(b)$\Rightarrow$(c) is clear since $C_n(\phi,U)\in\BV$ for every $n\in\N_+$. 

(c)$\Rightarrow$(a) If $C(\phi,U)$ is open, then $U/C(\phi,U)$ has finite dimension by Proposition~\ref{prop:lc properties}(d,e). Therefore, $H^*(\f,U)\leq\lim_{n\to \infty}\frac{1}{n}\dim\frac{U}{C(\phi,U)}=0$.  
\end{proof}

As a consequence, the topological entropy is always zero on discrete vector spaces, and so in particular on finite-dimensional vector spaces:

\begin{corollary}\label{cor:ent*discrete}
Let $\phi\colon V\to V$ be an endomorphism of a discrete vector space $V$. Then $\ent^*(\f)=0$.
\end{corollary}

On the other hand, for linearly compact vector spaces we can simplify the defining formula of the topological entropy as follows. Note that if $V$ is a linearly compact vector space, then $\BV=\{U\leq V\mid U\ \text{open}\}$; indeed, an open linear subspace of a linearly compact vector space is necessarily linearly compact by Proposition~\ref{prop:lc properties}(b).

\begin{lemma}\label{H*lc}
Let $V$ be a linearly compact vector space, $\phi:V\to V$ a continuous endomorphism and $U\in\BV$. Then $$H^*(\phi,U)=\lim_{n\to\infty}\frac{1}{n}\dim\frac{V}{C_n(\phi,U)}.$$
\end{lemma}
\begin{proof}
Since $U,C_n(\phi,U)\in\BV$, we have that $V/U$ and $V/C_n(\phi,U)$ have finite dimension by Proposition~\ref{prop:lc properties}(d,e). Then $\dim\frac{U}{C_n(\phi,U)}=\dim\frac{V}{C_n(\phi,U)}-\dim\frac{V}{U}$ and hence
$$H^*(\phi,U)=\lim_{n\to\infty}\frac{1}{n}\left(\dim\frac{U}{C_n(\phi,U)}-\dim\frac{V}{U}\right)=\lim_{n\to\infty}\frac{1}{n}\dim\frac{V}{C_n(\phi,U)},$$
so we have the thesis.
\end{proof}

By Corollary~\ref{cor:tdlc}, an l.l.c.\! vector space $V$ over a discrete finite field is in particular a totally disconnected locally compact abelian group. We end this section by relating $\ent^*$ to the topological entropy $h_{top}$.

\begin{proposition}\label{lem:tdlc}
Let $V$ be an l.l.c.\! vector space over a discrete finite field $\FF$ and $\phi:V\to V$ a continuous endomorphism. Then
$$h_{top}(\f)=\ent^*(\f)\cdot\log|\FF|.$$
\end{proposition}
\begin{proof}
As $\BV$ is a local basis at $0$ in $V$ contained in $\B_{gr}(V)$, and since also $H_{top}(\f,-)$ is monotone decreasing (see~\cite[Remark 4.5.1(b)]{DGBarxiv}), we have that
\begin{equation}
h_{top}(\f)=\sup\{H_{top}(\f,U)\mid U\in\BV\}.
\end{equation}
Now, for every $U\in\BV$, $$\left|\frac{U}{C_n(\phi,U)}\right|=|\FF|^{\dim\frac{U}{C_n(\phi,U)}},$$ and so
$$H_{top}(\f,U)=\lim_{n\to\infty} \frac{1}{n}\log \card{\frac{U}{C_n(\f,U)}}=\lim_{n\to\infty} \frac{1}{n}{\dim\frac{U}{C_n(\f,U)}}\log \card{\FF}=H^*(\f,U)\cdot\log \card{\FF}.$$
Thus, $h_{top}(\f)= \ent^*(\f)\cdot\log \card{\FF}$. 
\end{proof}

The previous result points out that, as long as we are dealing with l.l.c.\! vector spaces over discrete finite fields, the topological entropy $\ent^*$ turns to be a rescaling of the topological entropy $h_{top}$ and the most natural logarithm to compute the topological entropy $h_{top}$ is the one with base $\card{\FF}$.

\subsection{Fundamental properties}\label{fp-sec}

In this section we list the general properties and examples concerning the topological entropy $\ent^*$.

\begin{lemma}\label{lem:cotraj}
Let $V$ be an l.l.c.\! vector space, $\f\colon V\to V$ a continuous endomorphism, $W$ a closed $\f$-invariant linear subspace of $V$ and $\overline\f\colon V/W\to V/W$ the continuous endomorphism induced by $\f$. For $n\in\N_+$ and $U\in\BV$ and 
$$C_n(\f\restriction_W,U\cap W)=C_n(\f,U)\cap W\quad\text{and}\quad C_n\Big(\overline\f,\frac{U+W}{W}\Big)=\frac{C_n(\f,U+W)}{W}.$$
\end{lemma}
\begin{proof} 
Let $n\in\N_+$ and $U\in\BV$. Then $\f\restriction^{-n}_W(U\cap W)=\f^{-n} U\cap W$ and $\overline\f^{-n}(\frac{U+W}{W})=\frac{\f^{-n}(U+W)+W}{W}=\frac{\f^{-n}(U+W)}{W}$. The thesis follows form these equalities.
\end{proof}

We collect in the next result all the typical properties of an entropy function, that are satisfied by the topological entropy $\ent^*$.

\begin{proposition}\label{prop:basic prop}
Let $V$ be an l.l.c.\! vector space and $\phi\colon V \to V$ a continuous endomorphism.
\begin{enumerate}[(a)]
\item (Invariance under conjugation) If $\alpha\colon V\to W$ is a topological automorphism of l.l.c.\! vector spaces, then $\ent^*(\alpha\phi\alpha^{-1})=\ent^*(\phi)$.
\item (Monotonicity) If $W$ is a closed $\f$-invariant linear subspace of $V$ and $\overline\f\colon V/W\to V/W$ the continuous endomorphism induced by $\f$, then $\ent^*(\f)\geq \max\{\ent^*(\f\restriction_W),\ent^*(\overline\f)\}$. (If $W$ is open, then $\ent^*(\f)=\ent^*(\f\restriction_W)$.)
\item (Logarithmic law) If $k\in\N$, then $\ent^*(\phi^k)=k \cdot \ent^*(\phi)$.
\item (weak Addition Theorem) If $V = V_1 \times V_2$ for some l.l.c.\! vector spaces $V_1,V_2$, and $\phi= \phi_1 \times \phi_2: V \to V$ for some continuous endomorphisms $\phi_i\colon V_i\to V_i$, $i=1,2$, then $\ent^*(\phi) = \ent^*(\phi_1) + \ent^*(\phi_2)$. 
\item (Continuity on inverse limits) Let $\{W_i\mid i\in I\}$ be a directed system (under inverse inclusion) of closed $\f$-invariant linear subspaces of $V$. If $V=\varprojlim V/W_i$, then $\ent^*(\phi)=\sup_{i\in I} \ent^*(\overline\phi_{W_i})$, where any $\overline\f_{W_i}\colon V/W_i\to V/W_i$ is the continuous endomorphism induced by $\f$.
\end{enumerate}
\end{proposition}
\begin{proof} 
(a) Let $U\in\B(W)$ and $n\in\N_+$. Since $C_n(\alpha\f\alpha^{-1},U)=\alpha(C_n(\f,\alpha^{-1}U))$, it follows that
$$\dim \frac{U}{C_n(\alpha\f\alpha^{-1},U)}=\dim \frac{\alpha(\alpha^{-1}U)}{\alpha C_n(\f,\alpha^{-1}U)}=\dim \frac{\alpha^{-1}U}{C_n(\f,\alpha^{-1}U)}.$$
Hence $H^*(\alpha\phi\alpha^{-1},U)=H^*(\phi,\alpha^{-1}U)$. Since $\alpha$ is a topological isomorphism and $U\in\B(W)$ if and only if $\alpha^{-1} U\in\BV$, $\alpha$ induces a bijection between $\B(W)$ and $\BV$. Thus, $\ent^*(\alpha\phi\alpha^{-1})=\ent^*(\phi)$.

(b) Let $U\in\BV$. Lemma~\ref{lem:cotraj} yields $C_n(\phi\restriction_W,U\cap W)=C_n(\phi,U)\cap W$ for every $n\in\N_+$, so
$$\dim\frac{U\cap W}{C_n(\f\restriction_W,U\cap W)}\leq\dim\frac{U}{C_n(\f,U)}.$$ 
Hence, $H^*(\f\restriction_W,U\cap W)\leq H^*(\f,U)$. By Equation~\eqref{eq:basis}, $\B(W)=\{U\cap W\mid U\in\BV\}$, so we conclude that $\ent^*(\f)\geq\ent^*(\f\restriction_W)$. 
If $W$ is open, then $\B(W)$ is a neighborhood basis at $0$ in $V$, and so $\ent^*(\f)=\ent^*(\f\restriction_W)$ follows by Corollary~\ref{cor:coinitial}.

Let $U\in\BV$.  For every $n\in\N_+$, Lemma~\ref{lem:cotraj} implies that $C_n\Big(\overline\f,\frac{U+W}{W}\Big)=\frac{C_n(\f,U+W)}{W}$. Since moreover $W\leq C_n(\phi,U+W)\leq U+W$, we have that 
$$\frac{U+W}{C_n(\phi,U+W)}=\frac{U+W+C_n(\phi,U+W)}{W+C_n(\phi,U+W)}\cong\frac{U}{(W+C_n(\phi,U+W))\cap U}=\frac{U}{C_n(\phi,U+W)\cap U};$$
hence,
\begin{equation}\label{eq:iso1}
\bigslant{\frac{U+W}{W}}{C_n(\overline\f,\frac{U+W}{W})}\cong \frac{U+W}{C_n(\f, U+W)}\cong\frac{U}{C_n(\f, U+W)\cap U}.
\end{equation}
Since $C_n(\f,U)\leq C_n(\f,U+W)\cap U$, the the latter quotient vector space is actually a quotient of $\frac{U}{C_n(\f,U)}$. So,
\begin{equation}\label{eq:iso3}
\dim\Big( \bigslant{\frac{U+W}{W}}{C_n(\overline\f,\frac{U+W}{W})}\Big)\leq\dim\frac{U}{C_n(\f,U)},
\end{equation}
and hence $H^*(\overline\f,(U+W)/W)\leq H^*(\f,U)$.  Since $\B(V/W)=\left\{\frac{U+W}{W}\mid U\in\BV\right\}$ by Equation~\eqref{eq:basis}, we conclude that $\ent^*(\f)\geq\ent^*(\overline\f)$ .

(c) For $k=0$, $\ent^*(id_V)=0$ by Example~\ref{idex}. So let $k\in\N_+$ and $U\in\BV$. For every $n\in\N_+$ we have 
\begin{equation}\label{eq:ll1}
C_{nk}(\f,U)=C_n(\f^k,C_k(\f,U))\quad\text{and}\quad C_{n}(\f,C_k(\f,U))=C_{n+k-1}(\f,U).
\end{equation}
Let $E=C_k(\f,U)\in\BV$. Hence, by Lemma \ref{lem:mono*} and Equation~\eqref{eq:ll1},
\begin{eqnarray*}
k\cdot H^*(\f,U)&\leq& k\cdot H^*(\f,E)=k\cdot \lim_{n\to\infty}\frac{1}{nk}\dim\frac{E}{C_{nk}(\f,E)}=\lim_{n\to \infty}\frac{1}{n}\dim\frac{E}{C_{(n+1)k-1}(\f,U)}\\
                                              &\leq& \lim_{n\to \infty}\frac{1}{n}\dim\frac{E}{C_{(n+1)k}(\f,U)}= \lim_{n\to \infty}\frac{1}{n}\dim\frac{E}{C_{n+1}(\f^k,E)}=H^*(\f^k,E)\leq\ent^*(\f^k).
\end{eqnarray*}
Consequently, $k\cdot \ent^*(\f)\leq \ent^*(\f^k)$. Conversely, as $C_{nk}(\f,U)\leq E\leq U$, Equation~\eqref{eq:ll1} together with Lemma \ref{lem:mono*} yields
\begin{eqnarray*}
\ent^*(\f)\geq H^*(\f,U)&=&\lim_{n\to\infty}\frac{1}{nk}\dim \frac{U}{C_{nk}(\f,U)}=\lim_{n\to\infty}\frac{1}{nk}\dim\frac{U}{C_n(\f^k,E)}\\
&\geq&\lim_{n\to\infty}\frac{1}{nk}\dim\frac{E}{C_n(\f^k,E)}=\frac{1}{k}\cdot H^*(\f^k,E)\geq \frac{1}{k}H^*(\f^k,U).
\end{eqnarray*} 
So, $k\cdot \ent^*(\f)\geq\ent^*(\f^k)$.

(d) Observe that $\B=\{U_1\times U_2\mid U_i\in\mathcal{B}(V_i),i=1,2\}\subseteq \BV$ is neighborhood basis at $0$ in $V$. For $U=U_1\times U_2\in\B$, we have that $C_n(\phi,U)=C_n(\phi_1,U_1)\times C_n(\phi_2,U_2)$ for every $n\in\N_+$; therefore, $$\frac{U}{C_n(\phi,U)}=\frac{U_1\times U_2}{C_n(\phi_1,U_1)\times C_n(\phi_2,U_2)}\cong \frac{U_1}{C_n(\phi_1,U_1)}\times \frac{U_2}{C_n(\phi_2,U_2)},$$
and so $H^*(\phi,U)=\lim_{n\to \infty}\frac{1}{n}\frac{U}{C_n(\phi,U)}=\lim_{n\to\infty}\frac{1}{n}\dim\left(\frac{U_1}{C_n(\phi_1,U_1)}\times \frac{U_2}{C_n(\phi_2,U_2)}\right)=H^*(\phi_1,U_1)+H^*(\phi_2,U_2)$.
By Corollary~\ref{cor:coinitial}, we conclude that $\ent^*(\phi)=\sup\{H^*(\phi,U)\mid U\in\B\}=\ent^*(\phi_1)+\ent^*(\phi_2)$.

(e) By item (b), $\ent^*(\f)\geq\sup_{i\in I}\ent^*(\overline\f_{W_i})$. Conversely, let $U\in\BV$.  We claim that there exists $k\in I$ such that $W_k\leq U$. 
In fact, since $U$ is open in $V$, there exists an open linear subspace $A$ belonging to the canonical neighborhood basis at $0$ in $\prod_{i\in I} V/W_i$ such that $A\cap V\leq U$. Namely,  let $\pi_{i}\colon V\to V/W_i$ be the canonical projections. Since each $V/W_i$ is linearly topologized by the quotient topology, there exists a finite family $\{U_j\mid j\in J\}$, with $J\subseteq I$, of open linear subspaces of $V$ such that $W_j\leq U_j$ 
for all $j\in J$ and $A=\prod_{i\in I} A_i$, where $A_j=\pi_jU_j$ for $j\in J$ and $A_i=\pi_iV$ for $i\in I\setminus J$. Since $\{W_i\mid i\in I\}$ is directed, there exists $k\in I$ such that $j\leq k$ for all $j\in J$. Thus, given the open (and so closed) linear subspace $U_k=\bigcap_{j\in J} U_j + W_k\leq U_j$ ($j\in J$), we have that $\varprojlim\pi_iU_k$ is closed in $V$ since each $\pi_i U_k$ is open (and so closed) in $V/W_i$. Moreover, $U_k$ is a closed linear subspace of $\varprojlim\pi_i U_k$, thus by \cite[Lemma~1.1.7]{profinite}, it follows that 
$$U_k=\varprojlim_{i\in I}\pi_i U_k\leq A\cap V\leq U;$$ therefore $W_k\leq U$.

Thus, $H^*(\f,U)=H^*(\f, \pi_{k} U)$ by Lemma~\ref{lem:cotraj}, and hence $$H^*(\f,U)\leq\ent^*(\overline\f_{W_k})\leq\sup_{i\in I}\ent^*(\overline\f_{W_i}),$$ that concludes the proof.
\end{proof}

The Monotonicity property proved in item (b) of Proposition~\ref{prop:basic prop} can be improved to the following inequality, which is ``half'' of the Addition Theorem.

\begin{lemma}\label{semient}
Let $V$ be an l.l.c.\! vector space, $\f\colon V\to V$ a continuous endomorphism, $W$ a closed $\f$-invariant subspace of $V$ and $\overline\f\colon V/W\to V/W$ the continuous endomorphism induced by $\f$. Then $\ent^*(\f)\geq\ent^*(\f\restriction_W)+\ent^*(\overline\f)$.
\end{lemma}
\begin{proof}
By Lemma~\ref{lem:cotraj},
\begin{eqnarray}\label{eq:iso4}
\frac{U}{C_n(\f,U)+U\cap W}&\cong&\bigslant{\frac{U}{C_n(\f,U)}}{\frac{C_n(\f,U)+U\cap W}{C_n(\f,U)}}\nonumber\\
&\cong&\bigslant{\frac{U}{C_n(\f,U)}}{\frac{U\cap W}{C_n(\f,U)\cap U\cap W}}\cong \bigslant{\frac{U}{C_n(\f,U)}}{\frac{U\cap W}{C_n(\f\restriction_W,U\cap W)}}.
\end{eqnarray}
Since $C_n(\f,U)+U\cap W\leq  C_n(\f,U+W)\cap U$, Equations~\eqref{eq:iso1} and \eqref{eq:iso4} yield that
\begin{eqnarray}
\dim\Big(\bigslant{\frac{U+W}{W}}{C_n(\overline\f,\frac{U+W}{W})}\Big)&=&\dim\frac{U}{C_n(\phi,U+W)\cap U}\leq\dim\frac{U}{C_n(\phi,U)+U\cap W} \nonumber\\
&=&\dim\frac{U}{C_n(\f,U)}-\dim\frac{U}{C_n(\f\restriction_W,U\cap W)}.\nonumber
\end{eqnarray}
By Equation~\eqref{eq:basis} we have the thesis.
\end{proof}

The following provides the main examples in the theory of entropy functions, that is, the Bernoulli shifts.

\begin{example}\label{bernoulli}
\begin{enumerate}[(a)]
\item Let $V=V_d\times V_c$, where $$V_d=\bigoplus_{n=-\infty}^0 \K\quad\text{and}\quad V_c=\prod_{n=1}^\infty \K,$$ be endowed with the topology inherited from the product topology of $\prod_{n\in\Z}\K$. Then $V_c$ is linearly compact and $V_d$ is discrete.
In particular, $V_c$ can be identified with $0\times V_c\in\BV$. 

The \emph{left Bernoulli shift} is 
$${}_\K\beta\colon V\to V, \quad (x_n)_{n\in\Z}\mapsto (x_{n+1})_{n\in\Z}.$$
The \emph{right Bernoulli shift} is 
$$\beta_\K\colon V\to V, \quad (x_n)_{n\in\Z}\mapsto (x_{n-1})_{n\in\Z}.$$
Clearly, $\beta_\K$ and ${}_\K\beta$ are topological automorphisms of $V$, and $\beta^{-1}_\K={}_\K\beta$.

 For every $k\in\N_+$, let $U_k=0\times\prod_{n=k}^\infty \K\in\BV$, and consider $\ca{B}_f(V_c)=\{U_k\mid k\in\N_+\}\subseteq\BV$. Since $V_c\in\BV$, the family $\B_f(V_c)$ is a neighborhood basis at $0$ in $V$ contained in $\BV$. Thus, by Corollary~\ref{cor:coinitial}, for $\phi\in\{\beta_\K,{}_\K\beta\}$, $$\ent^*(\phi)=\sup\{H^*(\phi,U)\mid U\in\B_f(V_c)\}.$$  

Let $k\in\N_+$ and consider $U_k\in\B_f(V_c)$.  As $\beta_\K(U_k)\leq U_k\leq{}_\K\beta(U_k)$, we have the chain
$$\ldots\leq \beta_\K^n(U_k)\leq\ldots\leq\beta_\K(U_k)\leq U_k\leq{}_\K\beta(U_k)\leq\ldots\leq{}_\K\beta^n(U_k)\leq\ldots.$$
Therefore, for every $n\in\N_+$, $C_n(\beta_\K,U_k)=U_k$ and $C_n({}_\K\beta,U_k)=\beta_\K^{n-1}(U_k)=U_{k+n-1}$.
Thus, $H^*(\beta_\K,U_k)=0$ and, by Proposition~\ref{prop:exist}, we can compute $$H^*({}_\K\beta,U_k)=\inf_{n\in\N_+}\dim\frac{C_n({}_\K\beta,U_k)}{C_{n+1}({}_\K\beta,U_k)}=1.$$ 
We conclude that $\ent^*({}_\K\beta)=1$ and $\ent^*(\beta_\K)=0$ by Corollary~\ref{cor:coinitial}. 

\item Let now $F$ be a finite dimensional vector space and let $V=V_d\times V_c$, with
$$V_d=\bigoplus_{n=1}^\infty F\quad\mbox{and}\quad V_c=\prod_{n=-\infty}^0 F,$$
be endowed with the topology inherited from the product topology of $\prod_{n\in\Z}F$, so $V_d$ is discrete and $V_c$ is linearly compact.

The \emph{left Bernoulli shift} is
$${}_F\beta\colon V\to V, \quad (x_n)_{n\in\Z}\mapsto (x_{n+1})_{n\in\Z},$$
while the \emph{right Bernoulli shift} is
$$\beta_F\colon V\to V, \quad (x_n)_{n\in\Z}\mapsto(x_{n-1})_{n\in\Z}.$$
Clearly, $\beta_F$ and ${}_F\beta$ are topological automorphisms such that ${}_F\beta^{-1}=\beta_F.$

It is possible, slightly modifying the computations in item (a), to find that $$\ent({}_F\beta)=\dim F\quad \text{and}\quad \ent(\beta_F)=0.$$
\end{enumerate}
\end{example}

By the latter example, it is clear that for a topological automorphism $\phi:V\to V$ of an l.l.c.\! vector space $V$, in general $\ent^*(\phi^{-1})$ do not coincide with $\ent^*(\phi)$; consequently, property (c) of Proposition~\ref{prop:basic prop} cannot be extended to any integer $k\in\Z$.
In the last part of this section we find the precise relation between $\ent^*(\phi^{-1})$ and $\ent^*(\phi)$ and we deduce that equality holds in case $V$ is linearly compact.

\smallskip
In analogy to the classical modulus for topological automorphisms of locally compact groups, we define the \emph{dimension modulus} of $V$ by
$$\Delta_{\dim}\colon \Aut(V)\to \Z,\quad \f\mapsto\Delta_{\dim}(\phi,U)\quad \text{for}\ U\in\BV,$$ where
$$\Delta_{\dim}(\phi,U)=\dim\frac{\f U}{U\cap\f U}-\dim\frac{U}{U\cap\f U}.$$
In the next lemma we verify that this definition is well-posed, in fact it does not depend on the choice of $U\in\BV$.

\begin{lemma}\label{lem:mod}
Let $V$ be an l.l.c.\! vector space, $\f\colon V\to V$ a topological automorphism and $U_1,U_2\in\BV$. Then $\Delta_{\dim}(\phi,U_1)=\Delta_{\dim}(\phi,U_2)$.
\end{lemma}
\begin{proof} 
Since $\BV$ is closed under taking finite intersections, in order to prove the thesis we can restrict to the case $U_1\cap U_2$.
Since $U_1\cap\f U_1\leq\f U_1\leq \f U_2$ and $U_1\cap\f U_1\leq U_1\leq U_2$, we have
\begin{eqnarray} 
\dim\frac{\f U_1}{U_1\cap\f U_1}&=&\dim\frac{\f U_2}{U_1\cap\f U_1}-\dim\frac{\f U_2}{\f U_1},\nonumber\\
\dim\frac{U_1}{U_1\cap\f U_1}&=&\dim\frac{U_2}{U_1\cap\f U_1}-\dim\frac{U_2}{U_1}.\nonumber
\end{eqnarray}
Since $\f$ is an automorphism,  $\dim\frac{\f U_2}{\f U_1}=\dim\frac{U_2}{U_1}$, and so 
\begin{equation*}
\Delta(\phi,U_1)=\dim\frac{\f U_1}{U_1\cap\f U_1}-\dim\frac{U_1}{U_1\cap\f U_1}=\dim\frac{\f U_2}{U_1\cap\f U_1}-\dim\frac{U_2}{U_1\cap\f U_1}.
\end{equation*}
Analogously, since $U_1\cap \f U_1\leq U_2\cap\f U_2\leq \f U_2$ and $U_1\cap \f U_1\leq U_2\cap\f U_2\leq U_2$, 
\begin{eqnarray} 
\dim\frac{\f U_2}{\f U_1\cap\phi U_2}&=&\dim\frac{\f U_2}{U_1\cap\f U_1}-\dim\frac{U_2\cap\phi U_2}{U_1\cap\f U_1},\nonumber\\
\dim\frac{U_2}{U_2\cap \phi U_2}&=&\dim\frac{U_2}{U_1\cap\f U_1}-\dim\frac{U_2\cap \phi U_2}{U_1\cap\f U_1}.\nonumber
\end{eqnarray}
Since $\f$ is an automorphism,  $\dim\frac{\f U_2}{\f U_1}=\dim\frac{U_2}{U_1}$, and so 
\begin{equation*}
\Delta(\phi,U_2)=\dim\frac{\f U_2}{U_2\cap\f U_2}-\dim\frac{U_2}{U_2\cap\f U_2}=\dim\frac{\f U_2}{U_1\cap\f U_1}-\dim\frac{U_2}{U_1\cap\f U_1}.
\end{equation*}
Therefore, $\Delta(\phi,U_1)=\Delta(\phi,U_2)$ as required.
\end{proof}

The following result provides the precise relation between $\ent^*(\phi^{-1})$ and $\ent^*(\phi)$, it is inspired by its analogue for totally disconnected locally compact groups provided in \cite[Proposition 3.2]{GB}.

\begin{proposition} 
Let $V$ be an l.l.c.\! vector space and $\f:V\to V$ a topological automorphism of $V$. If $U\in\BV$, then
$H^*(\f^{-1},U)=H^*(\f,U)-\Delta_{\dim}(\f)$.
Hence, $$\ent^*(\f^{-1})=\ent^*(\f)-\Delta_{\dim}(\f).$$
\end{proposition}
\begin{proof} 
Since  $C_n(\f^{-1},U)=\f^nC_n(\f,U)$ and $C_{n+1}(\f,U)=C_n(\f,U)\cap\f^{-1}C_n(\f,U)$ for every $n\in\N_+$,
\begin{eqnarray}
\dim\frac{C_n(\f,U)}{C_{n+1}(\f,U)}&=&\dim\frac{C_n(\f,U)}{C_n(\f,U)\cap\f^{-1}C_n(\f,U)} \nonumber\\ 
&=& \dim\frac{\f^{n+1}C_n(\f,U)}{\f^{n+1}C_n(\f,U)\cap\f^nC_n(\f,U)}=\dim\frac{\f C_n(\f^{-1},U)}{\f C_n(\f^{-1},U)\cap C_n(\f^{-1},U)}.\nonumber
\end{eqnarray}
By Proposition~\ref{prop:exist} and Lemma~\ref{lem:mod}, it follows that 
\begin{eqnarray}\nonumber
H^*(\f,U)-H^*(\f^{-1},U)= \inf_{n\in\N_+} \left(\dim\frac{C_n(\f,U)}{C_{n+1}(\f,U)}-\dim\frac{C_n(\f^{-1},U)}{C_{n+1}(\f^{-1},U)}\right)=\\ \nonumber
=\inf_{n\in\N_+}\left(\dim\frac{\f C_n(\f^{-1},U)}{\f C_n(\f^{-1},U)\cap C_n(\f^{-1},U)}-\dim\frac{ C_n(\f^{-1},U)}{C_n(\f^{-1},U)\cap\f C_n(\f^{-1},U)}\right)=\Delta_{\dim}(\f),
\end{eqnarray}
since $C_n(\f^{-1},U)\in\BV$ for all $n\in\N_+$.
\end{proof}

If $V$ is a linearly compact vector space and $\phi:V\to V$ a topological automorphism, then clearly $\Delta_{\dim}(\phi)=0$. Therefore, we have the following direct consequence of the above proposition.

\begin{corollary}
Let $V$ be a linearly compact vector space and $\phi:V\to V$ a topological automorphism. If $U\in\BV$, then $H^*(\f^{-1},U)=H^*(\f,U)$.
Hence, $$\ent^*(\f^{-1})=\ent^*(\f).$$
\end{corollary}


\subsection{Limit-free Formula}\label{ss:lff}

This section is devoted to prove Proposition~\ref{prop:lff}, which is a formula for the computation of the topological entropy avoiding the limit in the definition. The proof follows the technique used in \cite[Proposition~3.9]{GBVirili}, which was developed by Willis in \cite{Willis}.

\begin{definition}\label{def:uplus}
Let $V$ be an l.l.c.\! vector space, $\f:V\to V$ a continuous endomorphism and $U\in\BV$. Let:
\begin{itemize}
\item[-] $U_0=U$;
\item[-] $U_{n+1}=U\cap\f U_n$, for all $n\in\N$;
\item[-] $U_+=\bigcap_{n\in\N}U_n$. 
\end{itemize}
\end{definition}
For every $n\in\N_+$, $U_n\geq U_{n+1}\geq U_+$; moreover, each $U_n$ is linearly compact and so is $U_+$ (see Proposition~\ref{prop:lc properties}(a,b,c)). 
Furthermore, it is possible to prove by induction that, for every $n\in\N$,
\begin{equation}\label{eq:Un}
U_n=\{u\in U\mid \exists v\in U,\ u=\f^n(v)\ \text{and}\ \f^j(v)\in U\ \forall j\in\{0,\ldots,n\}\}.
\end{equation}
Since $C_{n+1}(\f,U)=\{u\in U\mid \f^k(u)\in U\ \forall k\in\{0,\ldots,n\}\}$, it follows that
\begin{equation}\label{eq:cotrajUn}
\f^n C_{n+1}(\f,U)=U_n\quad\forall n\in\N.
\end{equation}
For every $n\in\N$, $U_n=C_{n+1}(\f^{-1},U)$ whenever $\f$ is also injective.

\smallskip
In the following result we collect the main properties of the linearly compact subgroup $U_+$ of an l.l.c.\! vector space $V$ for $U\in\BV$.

\begin{lemma}\label{prop:Uplus}
Let $V$ be an l.l.c.\! vector space, $\phi:V\to V$ a continuous endomorphism and $U\in\BV$. Then:
\begin{enumerate}[(a)]
\item $U_+$ is the largest linear subspace of $U$ such that $U_+\leq\f U_+$;
\item $U_+=U\cap\f U_+$;
\item  $\phi U_+/U_+$ has finite dimension.
\end{enumerate}
\end{lemma}
\begin{proof} 
(a) Since $U_{n+1}\leq U_n\leq \f U_{n-1}$ for all $n\in\N_+$, by applying Theorem~\ref{thm:warner}(a) to $U$ and the decreasing chain $\{U_n\}_{n\in\N}$, we have
$$U_+=\bigcap_{n\in\N} U_n \leq \bigcap_{n\in\N}\f U_n=\f\Big(\bigcap_{n\in\N}U_n\Big)=\f U_+.$$

Moreover, for every linear subspace $W$ of $V$ such that $W\leq U$ and $W\leq \f W$, it is possible to prove by induction that $W\leq U_n$ for all $n\in\N$, and so $W\leq U_+$.

(b) By construction,
$$U\cap\f U_+=\bigcap_{n\in\N}(U\cap\f U_n)=\bigcap_{n\in\N}U_{n+1}=U_+.$$

(c) Since $U_+=U\cap \phi U_+$ by item (b), $U_+$ is open in $\phi U_+$, which is linearly compact. Then $\phi U_+/U_+$ has finite dimension by Proposition~\ref{prop:lc properties}(d,e).
\end{proof}

%

We are now in position to prove the Limit-free Formula.

\begin{proposition}[Limit-free Formula]\label{prop:lff}
Let $V$ be an l.l.c.\! vector space, $\f:V\to V$ a continuous endomorphism and $U\in\BV$. Then
\begin{equation}\label{eq:lff}
H^*(\f,U)=\dim\frac{\f U_+}{ U_+}\big.
\end{equation}
\end{proposition} 
\begin{proof}
Let $U\in\BV$. For every $n\in\N_+$, let
$$\gamma_n=\dim\frac{C_n(\phi,U)}{C_{n+1}(\phi,U)}.$$
By Proposition~\ref{prop:exist}, the sequence $\{\gamma_n\}_{n\in\N}$ is stationary, and $H^*(\f,U)=\gamma$ where $\gamma$ is the value of the stationary sequence $\{\gamma_n\}_{n\in\N}$ for $n\in\N_+$ large enough. Hence, it suffices to prove that 
\begin{equation}\label{lffeq}
\dim\frac{\f U_+}{U_+}=\gamma.
\end{equation}

Since $\f U_n+U$ is linearly compact for every $n\in\N$ by Proposition~\ref{prop:lc properties}(f), thus $\dim\frac{\f U_n+U}{U}$ is finite, being $U$ open, by Proposition~\ref{prop:lc properties}(d,e). Moreover, since $\f U_n\geq\f U_{n+1}$ for all $n\in\N$, the sequence of non-negative integers $\left\{\dim\frac{\f U_n+U}{U}\right\}_{n\in\N}$ is decreasing, and so stationary. Thus, there exists $n_0\in\N$ such that, for every $n\geq n_0$,
$$\gamma=\dim\frac{C_n(\phi,U)}{C_{n+1}(\phi,U)}\quad \text{and}\quad \f U_n +U=\f U_{n_0} +U;$$
since $\{\phi U_n+U\}_{n\in\N}$ is a decreasing chain, $$\f U_{n_0} +U=\bigcap_{n\in\N}(\phi U_n+U).$$
Let $m\geq n_0.$ Theorem~\ref{thm:warner}(b) applied to the linearly compact linear subspace $U+\f U$ and the descending chain $\{U_n\}_{n\in\N}$ yields
\begin{equation*}\label{eq:sum}
\f U_m + U=\bigcap_{n\in\N} (\f U_n+U)=\bigg(\bigcap_{n\in\N} \f U_n\bigg)+U=\f\bigg(\bigcap_{n\in\N} U_n\bigg)+U=\f U_++U.
\end{equation*}
As $U_+=U\cap\f U_+$ by Lemma~\ref{prop:Uplus}(b), we have
\begin{eqnarray*}
\dim\frac{\f U_+}{U_+}&=&\dim\frac{\f U_+}{U\cap\f U_+}=\dim\frac{U+\f U_+}{U}\\
&=&\dim\frac{U+\f U_m}{U}=\dim\frac{\f U_m}{U\cap\f U_m}=\dim\frac{\f U_m}{U_{m+1}}.
\end{eqnarray*}
Equation~\eqref{eq:cotrajUn} now gives $\f^{m+1}C_{m+1}(\f,U)=\f U_m$, so there exists a surjective homomorphism
$$\varphi:C_{m+1}(\phi,U)\to\f U_m/U_{m+1},\quad x\mapsto \f^{m+1}(x)+U_{m+1},$$
such that $\ker\varphi=C_{m+2}(\phi,U)$. Hence, $$\frac{\phi U_m}{U_{m+1}}\cong \frac{C_{m+1}(\phi,U)}{C_{m+2}(\phi,U)}.$$
Finally, $\dim\frac{\f U_+}{U_+}=\dim\frac{\f U_m}{U_{m+1}}=\dim\frac{C_{m+1}(\phi,U)}{C_{m+2}(\phi,U)}=\gamma$, as required in Equation~\eqref{lffeq}.
\end{proof}

The following useful consequence of the Limit-free Formula is inspired by its analogue in the context of totally disconnected locally compact groups. Here we adapt the proof of \cite[Proposition 3.11]{GBVirili} to our context for reader's convenience.

\begin{corollary}\label{cor:ent*lc}
Let $V$ be an l.l.c.\! vector space and $\phi:V\to V$ a continuous endomorphism. Then
$$\ent^*(\phi)=\sup\left\{\dim\frac{\phi M}{M}\mid M\leq\phi M\leq V,\ \text{$M$ linearly compact},\  \dim\frac{\phi M}{M}<\infty\right\}=:s.$$
\end{corollary}
\begin{proof}
By Proposition~\ref{prop:lff}, $\ent^*(\phi)\leq s$. To prove the converse inequality, let $M$ be a linearly compact linear subspace of $V$ such that $M\leq \phi M$ and $\phi M/M$ has finite dimension. By Proposition~\ref{prop:lc properties}(a,d,e), this implies that $M$ is open in $\phi M$, since $M$ is closed and $\phi M$ is linearly compact, namely, $M\in\B(\phi M)$. By Equation~\eqref{eq:basis}, there exists $U\in\B(V)$ such that $M=U\cap\phi M$. As $M\leq\f M$ and $M\leq U$, we deduce that $M\leq U_+$ by Lemma~\ref{prop:Uplus}(a). Therefore $\f M\leq\f U_+$, and so
$$\ent^*(\f)\geq\dim\frac{\phi U_+}{ U_+}=\dim\frac{\phi U_+}{U\cap\phi U_+}\geq\dim\frac{\phi U_+\cap\f M}{U\cap\phi U_+\cap\f M}=\dim\frac{\phi M}{U\cap \phi M}=\dim\frac{\phi M}{M},$$
since  $U_+=U\cap\f U_+$ by Lemma~\ref{prop:Uplus}(b). Finally, it follows that $s\leq \ent^*(\phi)$.
\end{proof}

\section{Reductions for the computation of the topological entropy}

In this section we provide two reductions that can be used to simplify the computation of the topological entropy. The first one follows from the Limit-free Formula and shows that it is sufficient to consider linearly compact vector spaces. Once we restrict to linearly compact vector spaces, we make a second reduction to topological automorphisms.

\subsection{Reduction to linearly compact vector spaces}\label{s:rlc}

Let $V$ be an l.l.c.\! vector space and $\phi:V\to V$ a continuous endomorphism. By Theorem~\ref{thm:dec} we can assume that $V=V_c\oplus V_d$, where $V_c\in\BV$ and consequently $V_d$ is discrete. Let
\begin{equation*}
\iota_*\colon V_*\to V,\quad p_*\colon V\to V_*
\end{equation*}
with $*\in\{c,d\}$ be the canonical injections and projections, respectively. Accordingly, we may associate to $\f$ the following decomposition
\begin{equation*}
\f=
\begin{pmatrix}
  \f_{cc} & \f_{dc} \\
  \f_{cd} & \f_{dd}
\end{pmatrix},
\end{equation*}
where $\phi_{\bullet*}:V_\bullet\to V_*$ is the composition $\phi_{\bullet*}=p_*\circ\phi\circ\iota_\bullet$ for $\bullet, *\in\{c,d\}$.
Therefore, $\f_{\bullet*}$ is continuous being composition of continuous homomorphisms.

\begin{lemma}\label{lem:kerim} 
In the above notations, consider $\f_{cd}\colon V_c\to V_d$. Then 
$$\Im(\f_{cd})\in\B(V_d)\quad\text{and}\quad\ker(\f_{cd})\in\B(V_c)\subseteq \BV.$$
\end{lemma}
\begin{proof}
By Proposition~\ref{prop:lc properties}(c), $\Im(\f_{cd})$ is a linearly compact linear subspace of $V_d$. For $V_d$ is discrete, $\Im(\f_{cd})$ has finite dimension by Proposition~\ref{prop:lc properties}(d,e), so $\Im(\phi_{cd})\in\B(V_d)=\{F\leq V_d\mid \dim F<\infty\}$.

 As $\ker(\f_{cd})$ is a closed linear subspace of $V_c$, which is linearly compact, $\ker(\f_{cd})$ is linearly compact as well by Proposition~\ref{prop:lc properties}(b).  Thus $V_c/\ker(\f_{cd})\cong\Im(\f_{cd})$ is a finite dimesional linearly compact space, so $V/\ker(\f_{cd})$ is discrete by Proposition~\ref{prop:lc properties}(d,e). Consequently, $\ker(\f_{cd})$ is open in $V_c$, and so $\ker(\f_{cd})\in\B(V_c)$.
\end{proof}

We see now that in the above decomposition of $\phi$, the unique contribution to the topological entropy of $\phi$ comes from the ``linearly compact component'' $\phi_{cc}$.

\begin{proposition}\label{prop:rescc}
In the above notations, consider $\f_{cc}\colon V_c\to V_c$. Then $\ent^*(\f)=\ent^*(\f_{cc})$.
\end{proposition}
\begin{proof} 
By Lemma~\ref{lem:kerim}, $K=\ker(\f_{cd})\in\B(V_c)\subseteq\BV$. Thus, by Corollary~\ref{cor:coinitial},
\begin{eqnarray}
\ent^*(\f)&=&\sup\{H^*(\f,U)\mid U\in\B(K)\},\nonumber\\
\ent^*(\f_{cc})&=&\sup\{H^*(\f_{cc},U)\mid U\in\B(K)\}.\nonumber
\end{eqnarray}
For $U\in\B(K)$, as in Definition~\ref{def:uplus}, let 
\begin{eqnarray}
U_0=U\quad&\text{and}&\quad U^{cc}_0=U,\nonumber\\
U_{n+1}=U\cap\f U_{n}\quad&\text{and}&\quad U^{cc}_{n+1}=U\cap\f_{cc} U^{cc}_{n},\quad \text{for every}\ n\in\N,\nonumber\\ 
U_+=\bigcap_{n\in\N} U_n\quad&\text{and}&\quad U^{cc}_+=\bigcap_{n\in\N} U^{cc}_n.\nonumber
\end{eqnarray}
Proposition~\ref{prop:lff} implies that
\begin{eqnarray}
\ent^*(\f)&=&\sup\left\{\dim\frac{\f U_+}{U_+}\mid U\in\B(K)\right\},\nonumber\\
\ent^*(\f_{cc})&=&\sup\left\{\dim\frac{\f_{cc}U^{cc}_+}{U^{cc}_+}\mid U\in\B(K)\right\}.\nonumber
\end{eqnarray}
Since $U\leq K=\ker(\f_{cd})\leq V_c$, we deduce that $\phi U=\phi_{cc} U\leq V_c$. Then it is possible to prove by induction that $U_n=U^{cc}_n$ for all $n\in\N$; therefore, $U_+= U^{cc}_+$ and $\f U_+= \f_{cc}U^{cc}_+$.
Hence, $\ent^*(\phi)=\ent^*(\phi_{cc})$.
\end{proof}

\subsection{Reduction to topological automorphisms}\label{s:redauto}

For a continuous endomorphism $\phi:V\to V$ of a linearly compact vector space $V$, the \emph{surjective core} of $\f$ is
\begin{equation*}
S_{\f}=\bigcap_{n\in\N}\f^nV.
\end{equation*}
\begin{lemma}\label{fact:surjcore}
Let $V$ be a linearly compact vector space and $\phi:V\to V$ a continuous endomorphism. Then:
\begin{enumerate}
\item[(a)] $S_{\f}$ is a closed linear subspace of $V$;
\item[(b)] $\f(S_{\f})=S_{\f}$;
\item[(c)] $U_+\leq S_\f$ for all $U\in\BV$.
\end{enumerate}
\end{lemma}
\begin{proof} 
(a) is an easy consequence of Proposition~\ref{prop:lc properties}(a,c), while Theorem~\ref{thm:warner}(a) implies (b), and item (c) follows by the definition of $U_+$.
\end{proof}

Thus, $S_{\f}$ is a closed $\f$-stable linear subspace of $V$, and so $\f\restriction_{S_{\f}}\colon S_{\f}\to S_{\f}$ is surjective. The following result shows that $\ent^*(\f)=\ent^*(\f\restriction_{S_\f})$. Moreover, by definition $S_{\f}$ turns to be the largest closed $\phi$-stable linear subspace of $V$.

\begin{proposition}\label{thm:redepi}
Let $V$ be a linearly compact vector space and $\f:V\to V$ a continuous endomorphism. Then
$$\ent^*(\f)=\ent^*(\f\restriction_{S_{\f}}).$$
\end{proposition}
\begin{proof}
By Proposition~\ref{prop:basic prop}(b), $\ent^*(\f)\geq\ent^*(\f\restriction_{S_{\f}})$. To prove the converse inequality, let $U\in\BV$. By Lemma~\ref{prop:Uplus}(c) and Lemma~\ref{fact:surjcore}(c), $U_+$ is linearly compact, $U_+\leq\f U_+\leq S_\f$ and $\f U_+/U_+$ has finite dimensione. Thus,
\begin{equation}\label{eq:redepi}
H^*(\f,U)=\dim\frac{\f U_+}{U_+}\leq \ent^*(\f\restriction_{S_{\f}}),
\end{equation}
by Proposition~\ref{prop:lff} and Corollary~\ref{cor:ent*lc}. Consequently, $\ent^*(\f)\leq\ent^*(\f\restriction_{S_{\f}}).$
\end{proof}

Let $V$ be a linearly compact vector space and $\phi:V\to V$ a continuous endomorphism. Let $\caL V$ denote the inverse limit $\varprojlim(V_n,\f)$ of the inverse system $(V_n,\f)_{n\in\N}$, where $V_n=V$ for all $n\in\N$:
\begin{equation}\label{eq:invsys}
\cdots\stackrel{\f}{\longrightarrow}V_n\stackrel{\f}{\longrightarrow}V_{n-1}\stackrel{\f}{\longrightarrow}\cdots\stackrel{\f}{\longrightarrow}V_1\stackrel{\f}{\longrightarrow} V_0.
\end{equation} 
In other words,
\begin{equation}
\caL V=\left\{(x_n)_{n\in\N}\in\prod_{n\in\N}V_n\mid x_n=\phi(x_{n+1})\ \forall n\in\N\right\},
\end{equation}
endowed with the topology inherited from the product topology of $\prod_{n\in\N}V_n$. Since $\caL V$ is a closed linear subspace of the direct product (see \cite[Lemma~1.1.2]{profinite}), $\caL V$ is linearly compact as well by Proposition~\ref{prop:lc properties}(c). Let $\iota:\mathcal LV\to \prod_{n\in\N}V_n$ be the canonical embedding.


The natural continuous endomorphism 
$$\prod\f\colon\prod_{n\in\N}V_n\to\prod_{n\in\N}V_n,\quad (x_n)_n\mapsto (\f(x_n))_n$$ induces a continuous endomorphism $\caL\phi:\caL V\to\caL V$ making the following diagram commute
\begin{equation}\label{eq:phi}
\xymatrix{\prod_{n\in\N}V_n\ar[r]^{\prod\f}&\prod_{n\in\N}V_n\\
\caL V\ar[u]^\iota\ar[r]_{\caL\phi}&\caL V.\ar[u]_\iota}
\end{equation}

\begin{proposition}\label{prop:auto}
Let $V$ be a linearly compact vector space and $\phi:V\to V$ a continuous endomorphism. Then $\caL\phi\colon\caL V\to\caL V$ is a topological automorphism.
\end{proposition}
\begin{proof}
By construction, $\caL\phi$ is continuous and injective. Since $\caL V$ is linearly compact, it is sufficient to prove that $\caL\phi$ is surjective by Proposition~\ref{prop:omt}. Let $x=(x_n)_n\in\mathcal LV$, that is, $\phi(x_{n+1})=x_n$ for every $n\in\N$. Clearly, $x=\mathcal L\phi((x_{n+1})_n)$. 
\end{proof}

The next part of this section is devoted to prove that $\ent^*(\phi)=\ent^*(\caL\phi)$. To this end, for every $n\in\N$, let $p_n\colon \caL V\to V_n$ be the canonical projection given by the usual restriction
\begin{equation}
\xymatrix{\caL V\ar@{^{(}->}[d]_{\iota}\ar[r]^{p_n}&V_n\\
\prod_{n\in\N}V_n.\ar@{->>}[ur]&}
\end{equation}
 Let also $K_n=\ker p_n$, which is a closed $\caL\phi$-invariant linear subspace of $\caL V$.

\begin{lemma}\label{lem:barV ent}
Let $V$ be a linearly compact vector space, $\phi:V\to V$ a continuous endomorphism and $n\in\N$. Then, in the above notations:
\begin{itemize} 
\item[(a)] $p_n(\caL V)=S_\phi$;
\item[(b)] for $\alpha\colon\caL V/K_n\to p_n(\caL V)$ the topological isomorphism induced by $p_n$ and $\overline{\caL\phi}_n\colon \caL V/K_n\to \caL V/K_n$ the continuous endomorphism induced by $\caL\phi$, 
$$\overline{\caL\phi}_n=\alpha^{-1}\circ\phi\restriction_{S_\phi}\circ\alpha;$$
\item[(c)] $\ent^*(\overline{\caL\phi}_n)=\ent^*(\phi)$.
\end{itemize}
\end{lemma}
\begin{proof}
(a) Let $x=(x_n)_n\in\caL V$; then $\f^k(x_{n+k})=x_n$ for every $k\in\N$, and so $p_n(x)=x_n\in S_\f$.

Conversely, let $s\in S_\f$. By Lemma~\ref{fact:surjcore}(b), we can set
\begin{itemize}
\item[-] $x_i=\f^{n-i}(s)\in S_\f$ for $i\in\{0,\ldots,n\}$;
\item[-] $x_{i}\in S_\f$ such that $\f(x_{i})=x_{i-1}$ for $i>n$.
\end{itemize}
Thus,  $x=(x_i)_i\in\caL V$ and $p_n(x)=s$, as required.

(b) By construction, $K_n$ is a closed $\caL\f$-invariant linear subspace of $\caL V$ for every $n\in \N$. Moreover, $\alpha$ is a topological isomorphism by Proposition~\ref{prop:omt}. Let $x+K_n\in\caL V/K_n$. By (a),
\begin{eqnarray}\nonumber
\alpha^{-1}(\f\restriction_{S_\f}(\alpha(x+K_n)))&=&\alpha^{-1}(\f\restriction_{S_\f}(p_n(x)))=\alpha^{-1}(\f(p_n(x)))=\\ \nonumber
&=&\alpha^{-1}(p_n(\caL\phi(x)))=\caL\phi(x)+K_n=\overline{\caL\phi}_n(x+K_n).
\end{eqnarray}

(c) Since $\alpha$ is a topological isomorphism, (c) is an easy consequence of (b) by Proposition~\ref{prop:basic prop}(a) and Proposition~\ref{thm:redepi}.
\end{proof}

The next result shows that for the computation of the topological entropy in the case of linearly compact vector spaces one can reduce to topological automorphisms.

\begin{theorem}\label{thm:redauto}
Let $V$ be a linearly compact vector space and $\f\colon V\to V$ a continuous endomorphism. 
Then $$\ent^*(\caL\phi)=\ent^*(\phi).$$
\end{theorem}
\begin{proof} 
Since $\bigcap_{n\in\N}K_n=0$, we have that the canonical map $\rho\colon \caL V\to\varprojlim\caL V/K_n$ is an injective continuous homomorphism of linearly compact spaces. Since $\rho(\caL V)$ is linearly compact, then $\rho$ is also surjective (see \cite[Lemma~1.1.7]{profinite}). By Proposition~\ref{prop:omt}, we conclude that
\begin{equation}\label{eq:invlim1}
\caL V\cong_{top} \varprojlim \caL V/K_n.
\end{equation}
By the invariance under conjugation, the latter identification preserves the topological entropy.
Now Proposition~\ref{prop:basic prop}(a,e) and Lemma~\ref{lem:barV ent}(c) give $\ent^*(\mathcal L\phi)=\sup_{n\in\N}\ent^*(\overline{\mathcal L\phi}_n)=\ent^*(\phi)$.
\end{proof}

A {\it flow} in the category $\LC$ is a pair $(V, \f)$, where $V$ is a linearly compact vector space and $\f \colon V\to V$ is a continuous endomorphism. If $(V, \f )$ and $(W, \psi)$ are flows in $\LC$, then a morphism of flows from $(V, \f )$ to $(W, \psi)$ is a continuous homomorphism $h\colon V\to W$ such that $h \circ\f = \psi\circ h$. We let $\mathrm{Flow}(\LC)$ denote the resulting category of flows in $\LC$. Clearly, it is well-defined a functor
\begin{equation}\label{functorL}
\caL\colon\mathrm{Flow}(\LC)\to\mathrm{Flow}(\LC)
\end{equation}
 given by $\caL(V,\f)=(\caL V,\caL \f)$ and $\caL(h)\colon \caL V\to \caL W$ is the continuous homomorphism induced by the following morphism of inverse systems
$$\xymatrix{
\cdots\ar[r]&V_n\ar[r]^\f\ar[d]^h&\cdots\ar[r]&V_1\ar[r]^\f\ar[d]^h&V_0\ar[d]^h\\ 
\cdots\ar[r]&W_n\ar[r]^\psi&\cdots\ar[r]&W_1\ar[r]^\psi&W_0,}$$ 
namely $\caL h((v_n)_n)=(h(v_n))_n\in\caL W$ for every $(v_n)_n\in\caL V.$

\smallskip
We conclude this section by showing that the functor $\caL$ preserves the topological extensions in the sense of  Proposition~\ref{prop: L exact}. Let $\f\colon V\to V$ be a continuous endomorphism of a linearly compact space $V$. For a closed $\f$-invariant linear subspace $W$ of $V$, consider the diagram
\begin{equation}
\xymatrix{
0\ar[r]&W\ar[r]\ar[d]^{\f\restriction_W}&V\ar[r]\ar[d]^{\f}&V/W\ar[r]\ar[d]^{\overline\f}&0\\
0\ar[r]&W\ar[r]&V\ar[r]&V/W.\ar[r]&0
}
\end{equation}
Thus, one constructs as above the following exact sequence of inverse systems of linearly compact spaces (see \eqref{eq:invsys})
\begin{equation}
\xymatrix{0\ar[r]&(W_n,\f\restriction_W)_{n\in\N}\ar[r]&(V_n,\f)_{n\in\N}\ar[r]&(V_n/W_n,\overline\f)_{n\in\N},\ar[r]&0}
\end{equation}
where $V=V_n$ and $W=W_n$ for every $n\in\N$.
Denote by
$$\caL W=\varprojlim(W_n,\phi\restriction_W),\quad \caL V=\varprojlim(V_n,\phi)\quad\text{and}\quad\caL (V/W)=\varprojlim(V_n/W_n,\overline{\f})$$
the corresponding inverse limits. Since the inverse limit functor in $\LC$ is exact (see Remark~\ref{rem:ab5}) we have the short exact sequence in $\LC$ (see \cite[pages 4-5]{profinite})
\begin{equation}
\xymatrix{0\ar[r]&\caL W\ar[r]&\caL V\ar[r]&\caL (V/W)\ar[r]&0.}
\end{equation}
In order to simplify the notation, we regard $\caL W$ as a closed linear subspace of $\caL V$ and, since $\caL(V/W)\cong_{top} \caL V/\caL W$, we identify $\caL (V/W)$ with the $\caL V/\caL W$. Since the topological entropy is invariant under conjugation by Proposition~\ref{prop:basic prop}(a), it is possible to easily verify the latter identification preserves the topological entropy. 

The linear subspace $\caL W$ turns out to be closed and $\caL\f$-invariant in $\mathcal LV$, so the following diagram 
\begin{equation}\label{eq:diag11}
\xymatrix{
0\ar[r]&\caL W\ar[r]\ar[d]^{(\caL\phi)\restriction_{\caL W}}&\caL V\ar[r]\ar[d]^{\caL \phi}&\caL(V/W)\ar[r]\ar[d]^{\overline{\caL\phi}}&0\\
0\ar[r]&\caL W\ar[r]&\caL V\ar[r]&\caL(V/W)\ar[r]&0
}
\end{equation}
commutes, where $\overline{\caL\phi}\colon \caL(V/W)\to\caL(V/W)$ is the continuous endomorphism induced by $\caL \phi$.

On the other hand, by restriction in view of Equation~\eqref{eq:phi} we have the commutative diagrams
\begin{equation}\label{eq:phi2}
\xymatrix{\prod_{n\in\N}W_n\ar[r]^{\prod\f\restriction_W}&\prod_{n\in\N}W_n\\
\caL W\ar[u]^{\iota'}\ar[r]_{\caL(\f\restriction_W)}&\caL W\ar[u]_{\iota'},}
\quad
\xymatrix{\prod_{n\in\N}(V_n/W_n)\ar[r]^{\prod\overline\f}&\prod_{n\in\N}(V_n/W_n)\\
\caL (V/W)\ar[u]^{\iota''}\ar[r]_{\caL\overline\phi}&\caL (V/W)\ar[u]_{\iota''},}
\end{equation}
where $\caL(\f\restriction_W)$ and $\caL\overline\f$ are both topological automorphisms by Proposition~\ref{prop:auto}.


\smallskip
We see now  how the functor $\mathcal L$ behaves under taking closed invariant linear subspaces and quotient vector spaces.

\begin{proposition}\label{prop: L exact} 
Let $V$ be a linearly compact vector space, $\f:V\to V$ a continuous endomorphism and $W$ a closed $\f$-invariant linear subspace of $V$. Then
$$(\caL\f)\restriction_{\caL W}=\caL(\f\restriction_W)\quad\text{and}\quad\overline{\caL\f}=\caL\overline\f,$$
where $\overline\phi:V/W\to V/W$ is the continuous endomorphism induced by $\phi$ and $\overline{\mathcal L\phi}:\mathcal LV/\mathcal LW\to \mathcal LV/\mathcal LW$ is the continuous endomorphism induced by $\mathcal L\phi$. In particular, $(\caL\f)\restriction_{\caL W}$ and $\overline{\caL\f}$ are topological automorphisms.

Moreover, $\ent^*(\f\restriction_W)=\ent^*((\caL\f)\restriction_{\caL W})$ and $\ent^*(\overline\f)=\ent^*(\overline{\caL\f})$.
\end{proposition}
\begin{proof}
Let $(x_n)_n\in\mathcal LW$, that is, $\f\restriction_W(x_{n+1})=x_n$ for all $n\in\N$. As $W\leq V$ and $\caL W\leq\caL V$,
$$\caL(\f\restriction_W)((x_n)_n)=\left(\f\restriction_W(x_n)\right)_n=(\f(x_n))_n=(\caL \f)\restriction_{\caL W}((x_n)_n).$$
Now let $(x_n+W)_{n}\in\caL(V/W).$ As $\caL (V/W)\cong\caL V/\caL W$, 
$$\caL\overline\f((x_n+W)_n)=\left(\f(x_n)+W\right)_n\cong (\f(x_n))_n+\caL W=\caL\f((x_n)_n)+\caL W=\overline{\caL\f}((x_n)_n+\caL W).$$

By Proposition~\ref{prop:auto}, it follows that $(\caL\phi)\restriction_{\caL W}$ and $\overline{\caL\f}$ are a topological automorphisms. Clearly, one deduces that $\caL W$ is a $\caL\f$-stable linear subspace of $\caL V$.

The last assertion follows from Theorem~\ref{thm:redauto}.
\end{proof}

\section{Addition Theorem}\label{s:AT}

This section is devoted to prove the Addition Theorem for the topological entropy (see Theorem~\ref{AT-intro}). 
We start by proving it for topological automorphisms.

\begin{proposition}\label{thm:ATauto}
Let $V$ be an l.l.c.\! vector space, $\f\colon V\to V$ a topological automorphism, $W$ a closed $\f$-stable linear subspace of $V$ and $\overline\phi:V/W\to V/W$ the topological automorphism induced by $\phi$. Then
$$\ent^*(\f)=\ent^*(\f\restriction_W)+\ent^*(\overline\f).$$
\end{proposition}
\begin{proof} 
By Lemma~\ref{semient}, we have $\ent^*(\f)\geq\ent^*(\f\restriction_W)+\ent^*(\overline\f)$. To prove the converse inequality, let $M$ be a linearly compact linear subspace of $V$ such that $M\leq\f M$ and $\f M/M$ has finite dimension. Then 
\begin{equation*}
\frac{\f M}{M+(\f M\cap W)}\cong\bigslant{\frac{\f M}{M}}{\frac{M+(\f M\cap W)}{M}}.
\end{equation*}
 Since $\frac{M+(\f M\cap W)}{M}\cong\frac{\f M\cap W}{\f M\cap W\cap M}=\frac{\f M\cap W}{M\cap W}$, we have that
\begin{equation}\label{eq:ent*lc3}
\dim\frac{\f M}{M}=\dim\frac{\f M\cap W}{M\cap W}+\dim\frac{\f M}{M+\f M\cap W}.
\end{equation}
Since $W$ is closed and $M$ is linearly compact, $M\cap W$ is linearly compact by Proposition~\ref{prop:lc properties}(a). Moreover, $\f(M\cap W)=\f M\cap \f W=\f M\cap W$, so $M\cap W$ has finite codimension in $\f(M\cap W)$. Corollary~\ref{cor:ent*lc} yields 
\begin{equation}\label{2}
\dim\frac{\f M\cap W}{M\cap W}\leq \ent^*(\f\restriction_W).
\end{equation}
On the other hand, $\pi M=\frac{M+W}{W}$ is linearly compact in $V/W$ by Proposition~\ref{prop:lc properties}(c). 
By the modular law, since $M\leq\f M$, , $M+(\f M\cap W)=\f M\cap(M+W)$. Therefore,
\begin{equation*}
\frac{\f M}{M+(\f M\cap W)}=\frac{\f M}{\f M\cap(M+W)}\cong\frac{\f M + W}{M+W},
\end{equation*}
Moreover, 
$$\pi(M)=\frac{M+W}{W}\leq\frac{\f M+W}{W}=\overline\f(\pi M).$$ 
Then
$$\dim\frac{\f M}{M+\f M\cap W}=\dim\frac{\overline\f(\pi M)}{\pi M},$$ 
and so, by Corollary~\ref{cor:ent*lc},
\begin{equation}\label{1}
\dim\frac{\f M}{M+\f M\cap W}\leq\ent^*(\overline\f).
\end{equation}
By Equations~\eqref{eq:ent*lc3}, \eqref{2} and \eqref{1}, we conclude that $$\dim\frac{\f M}{M}\leq\ent^*(\f\restriction_W)+\ent^*(\overline\f),$$
so Corollary~\ref{cor:ent*lc} yields the required inequality $\ent^*(\f)\leq\ent^*(\f\restriction_W)+\ent^*(\overline\f)$.
\end{proof}

A second step towards the proof of the Addition Theorem consists in proving it for linearly compact vector spaces.

\begin{proposition}\label{prop:ATlc}
Let $V$ be a linearly compact vector space, $\f\colon V\to V$ a continuous endomorphism, $W$ a closed $\f$-invariant linear subspace of $V$ and $\overline\f\colon V/W\to V/W$ the continuous endomorphism induced by $\f$. Then
$$\ent^*(\f)=\ent^*(\f\restriction_W)+\ent^*(\overline\f).$$
\end{proposition}
\begin{proof}
Consider the following short exact sequence of flows in $\LC$
\begin{equation*}
\xymatrix{
0\ar[r]&W\ar[r]\ar[d]^{\f\restriction_W}&V\ar[r]\ar[d]^{\f}&V/W\ar[r]\ar[d]^{\overline\f}&0\\
0\ar[r]&W\ar[r]&V\ar[r]&V/W\ar[r]&0.
}
\end{equation*}
By applying the functor $\caL$ (see \eqref{functorL} and \eqref{eq:diag11}), we obtain the commutative diagram
\begin{equation*}
\xymatrix{
0\ar[r]&\caL W\ar[r]\ar[d]^{(\caL\phi)\restriction_{\caL W}}&\caL V\ar[r]\ar[d]^{\caL \phi}&\caL(V/W)\ar[r]\ar[d]^{\overline{\caL\phi}}&0\\
0\ar[r]&\caL W\ar[r]&\caL V\ar[r]&\caL(V/W)\ar[r]&0,
}
\end{equation*}
where $\caL \phi:\caL V\to \caL V$ is a topological automorphism by Proposition~\ref{prop:auto}, and $\caL W$ is a closed $\caL\f$-stable linear subspace of $\caL$ by Proposition~\ref{prop: L exact}.  Therefore,
$$\ent^*(\f)=\ent^*(\caL\f)=\ent^*(\caL(\f\restriction_W))+\ent^*(\overline{\caL\phi})=\ent^*(\phi\restriction_W)+\ent^*(\overline\phi),$$
by Proposition~\ref{thm:ATauto},  Theorem~\ref{thm:redauto} and Proposition~\ref{prop: L exact}.
\end{proof}

We are now in position to prove the general statement of the Addition Theorem.

\begin{proof}[\bf Proof of Theorem~\ref{AT-intro}.]
Let $V_c\in\BV$ and $W_c=W\cap V_c$; then $W_c\in\B(W)$. By Theorem~\ref{thm:dec}, there exists a discrete linear subspace $W_d\leq W$ such that $W=W_c\oplus W_d$. Let $V_d\leq V$ such that $V=V_c\oplus V_d$ and $W_d\leq V_d$. Clearly, $V_d$ is a discrete subspace of $V$, since $V_c$ is open and $V_c\cap V_d=0$. By construction, the diagram
\begin{equation*}
\xymatrix{0\ar[r]&W_c\ar@/^2pc/[rrr]^{(\phi\restriction_{W})_{cc}} \ar[r]^{\iota^W_c}\ar@{^{(}->}[d]&W\ar[r]^{\f\restriction_W}&W\ar[r]^{p^W_c}&W_c\ar@{^{(}->}[d]\ar[r]&0\\
0\ar[r]&V_c\ar@/_2pc/[rrr]_{\phi_{cc}} \ar[r]^{\iota^V_c}&V\ar[r]^{\f}&V\ar[r]^{p^V_c}&V_c\ar[r]&0}
\end{equation*}
commutes, where $\iota^W_c, \iota^V_c,p^W_c,p^V_c$ are the canonical injections and projections of $W$ and $V$, respectively. This yields that $W_c$ is a closed $\f_{cc}$-invariant subspace of $V_c$ and that 
\begin{equation}\label{first}
(\f\restriction_W)_{cc}=\f_{cc}\restriction_{W_c}.
\end{equation}
Now, let $\pi\colon V\to V/W$ be the canonical projection and let $\overline V=V/W$. Let $\overline V_c=\pi(V_c)$ and $\overline V_d=\pi(V_d)$; then $\overline V_c\in\B(\overline V)$ since $\overline V_c$ is open and it is linearly compact by Proposition~\ref{prop:lc properties}(c), while $\overline V_d$ is discrete; moreover, $\overline V=\overline V_c\oplus\overline V_d$. 
The canonical continuous isomorphism $\alpha\colon V_c/W_c\to \overline V_c$ is a topological isomorphism by Proposition~\ref{prop:omt}; it makes the following diagram commute
\begin{equation*}
\xymatrix{
\overline V_c\ar@/^2pc/[rrr]^{\overline\f_{cc}} \ar[r]^{\iota^{\overline V}_c}&\overline V\ar[r]^{\overline\f}&\overline V\ar[r]^{p^{\overline V}_c}&\overline V_c\\
V_c/W_c \ar[rrr]^{\overline{\phi_{cc}}}\ar[u]^\alpha&&&V_c/W_c.\ar[u]^\alpha}
\end{equation*}
In other words,  $\alpha^{-1}\overline{\f}_{cc}\alpha=\overline{\f_{cc}}$ and so, by Proposition~\ref{prop:basic prop}(a),
\begin{equation}\label{second}
\ent^*(\overline{\f}_{cc})=\ent^*(\overline{\f_{cc}}).
\end{equation}
Finally, by using Equations~\eqref{first} and \eqref{second}, we obtain
\begin{equation*}\begin{split}
\ent^*(\phi)=\ent^*(\phi_{cc})&=\ent^*(\phi_{cc}\restriction_{W_c})+\ent^*(\overline{\phi_{cc}}) \\&=\ent^*((\phi\restriction_W)_{cc})+\ent^*(\overline\phi_{cc})=\ent^*(\phi\restriction_W)+\ent^*(\overline\phi),
\end{split}\end{equation*}
by Proposition~\ref{prop:rescc} and Proposition~\ref{prop:ATlc}.
\end{proof}

\section{Bridge Theorem}\label{s:bridge}

This section is devoted to prove that the topological entropy of a continuous endomorphism coincides with the algebraic entropy of the dual endomorphism with respect to the Lefschetz Duality. 

\begin{lemma}\label{B(G)}
Let $V$ be a l.l.c.\! vector space. Then $U\in\mathcal B(V)$ if and only if $U^\perp\in\mathcal B(\widehat V)$.
\end{lemma}
\begin{proof} 
Let $U\in\BV$. Since $U$ is linearly compact, $U^\perp$ is open by definition. Conversely, assume that $U^\perp$ is open. Since $U^\perp=\overline{U}^\perp$ by Lemma~\ref{perptop}(b), we may assume that $U$ is closed. Since $U^\perp$ is open in $\widehat V$, there exists a linearly compact linear subspace $W$ of $V$ such that $W^\perp\leq U^\perp$. Since $U$ and $W$ are closed linear subspaces of $V$, $U\leq W$ by Lemma~\ref{perptop}(a,d). Thus $U$ is a closed linear subspace of the linearly compact space $W$, and clearly $U$ is linearly compact by Proposition~\ref{prop:lc properties}(a).

Finally, since $U$ is open in $V$ if and only if the quotient $V/U$ is discrete, Remark~\ref{dualperp} implies that $U$ is open in $V$ if and only if $U^\perp\cong_{top} \widehat{V/U}$ is linearly compact. 
\end{proof}

\begin{lemma}\label{perp}
Let $V$ be a l.l.c.\! vector space, $\phi:V\to V$ a continuous endomorphism and $U\in\mathcal B(V)$.
Then $(\phi^{-n}(U))^\perp=(\widehat\phi)^n (U^\perp)$ for every $n\in\N_+$.
\end{lemma}
\begin{proof}
We prove the result for $n=1$, that is, 
\begin{equation}\label{1eq}
(\phi^{-1}(U))^\perp=\widehat\phi (U^\perp). 
\end{equation}
The proof for $n>1$ follows easily from this case noting that $(\widehat\phi)^n=\widehat{(\phi^n)}$. 

Let $W=U^\perp$; then $W\in\mathcal B(\widehat V)$ by Lemma~\ref{B(G)} and $U=W^\top$ by Lemma \ref{perptop}(d). We prove that 
\begin{equation}\label{2eq}
\phi^{-1}(W^\top)=(\widehat\phi(W))^\top,
\end{equation} 
that is equivalent to Equation~\eqref{1eq} by Lemma \ref{perptop}(d). So let $x\in \phi^{-1}(W^\top)$; equivalently, $\phi(x)\in W^\top$, that is $\chi(\phi(x))=0$ for every $\chi\in W$. This occurs precisely when $\widehat\phi(\chi)(x)=0$ for every $\chi\in W$, if and only if $x\in (\widehat\phi(W))^\top$. This chain of equivalences proves Equation~\eqref{2eq}.
\end{proof}

By applying the previous lemmas, we can now give a proof to the Bridge Theorem.

\begin{proof}[\bf Proof of Theorem~\ref{BT-intro}]
Let $U\in\BV$; so, $U^\perp\in\B(\widehat V)$ by Lemma~\ref{B(G)}. For $n\in\N_+$,  it follows from Lemma~\ref{lem:intsum} and Lemma~\ref{perp} that
$$C_n(\phi,U)^\perp=T_n(\widehat\phi,U^\perp).$$
Hence, in view of Lemma~\ref{A/B}, $U/C_n(\phi, U)\cong\widehat{U/C_n(\phi,U)}\cong T_n(\widehat\phi,U^\perp)/U^\perp$,
and so 
$$\dim \frac{U}{C_n(\phi,U)}=\dim\frac{T_n(\widehat\phi,U^\perp)}{U^\perp}.$$
Therefore, $H^*(\phi,U)=H(\widehat\phi,U^\perp)$.
By Lemma~\ref{B(G)}, we can conclude that $\ent^*(\phi)=\ent(\widehat\phi)$.
\end{proof}

As a consequence of the Addition Theorem for the topological entropy $\ent^*$ and the Bridge Theorem, we deduce now the Addition Theorem for the algebraic entropy $\ent$ proved in \cite{CGB}. 

\begin{corollary}
Let $V$ be an l.l.c.\! vector space, $\f:V\to V$ a continuous endomorphism and $W$ a $\phi$-invariant closed linear subspace of $V$. Then 
$$\ent(\f)=\ent(\f\restriction_W)+\ent(\overline\f).$$
\end{corollary}
\begin{proof}
Since $W^\perp$ is a closed $\widehat{\f}$-invariant linear subspace of $\widehat{V}$, consider the topological isomorphisms $\alpha:\widehat{V/W}\to W^\perp$ and $\beta:\widehat V/W^\perp\to \widehat W$ given by Equations~\eqref{alpha} and \eqref{beta}, respectively. It is possible to verify that the following diagrams commute
\begin{gather}\label{eq:add2}
\xymatrix{
W^\perp \ar[r]^{\widehat{\f}\restriction_{W^\perp}} &W^\perp\\
\widehat{V/W}\ar[u]^{\alpha} \ar[r]_{\widehat{\overline\f}}&\widehat{V/W\ar[u]_{\alpha} }
}
\qquad
\xymatrix{
\widehat{V}/W^\perp\ar[d]_\beta\ar[r]^{\overline{\widehat{\f}}} &\widehat{V}/W^\perp\ar[d]^{\beta} \\
\widehat{W}\ar[r]_{\widehat{\f\restriction_W}}&\widehat{W}.
}
\end{gather}
By Theorem~\ref{BT-intro} and Proposition~\ref{prop:basic prop}(a),
$$\ent(\f)=\ent^*(\widehat\f),\quad \ent(\f\restriction_W)=\ent^*(\widehat{\f\restriction_W})=\ent^*(\overline{\widehat\f})\quad\text{and}\quad\ent(\overline\f)=\ent^*(\widehat{\overline\f})=\ent^*(\widehat\f\restriction_{W^\perp}).$$
Then Theorem~\ref{AT-intro} yields
$$\ent(\f)=\ent^*(\widehat\f)=\ent^*(\overline{\widehat\f})+\ent^*(\widehat\f\restriction_{W^\perp})=\ent(\f\restriction_W)+\ent(\overline\f),$$
and this concludes the proof.
\end{proof}

Alternatively, one can deduce the Addition Theorem for the topological entropy $\ent^*$ from the Addition Theorem for the algebraic entropy $\ent$ and the Bridge Theorem.


\begin{thebibliography}{AAAA}

\bibitem{AKM} R.~L. Adler, A. G. Konheim, and M. H. McAndrew, \emph{Topological entropy}, Transactions of the American Mathematical Society 114 (1965): 309--319.
 
\bibitem{B} R. Bowen, \emph{Entropy for group endomorphisms and homogeneous spaces}, Trans. Amer. Math. Soc. 153 (1971): 401--414. 

\bibitem{CGB} I. Castellano, A. Giordano Bruno, \emph{Algebraic entropy for locally linearly compact vector spaces}, (2016) submitted.

\bibitem{conn} J. F. Conn, \emph{Non-abelian Minimal Closed Ideals of Transitive Lie Algebras}, Princeton University Press, 2014.


\bibitem{DGB-BT} D. Dikranjan, A. Giordano Bruno, \emph{The connection between topological and algebraic entropy}, Topology Appl. 159 (2012): 2980--2989.

\bibitem{DGB-lff} D. Dikranjan, A. Giordano Bruno, \emph{Limit free computation of entropy}, Rend. Istit. Mat. Univ. Trieste 44 (2012): 297--312.

\bibitem{DGBarxiv} D. Dikranjan, A. Giordano Bruno, \emph{Topological entropy and algebraic entropy for group endomorphisms}, Proceedings ICTA2011 Islamabad, Pakistan July 4-10 2011 Cambridge Scientific Publishers (2012): 133--214. 

\bibitem{GBD} D. Dikranjan, A. Giordano Bruno, \emph{The Bridge Theorem for totally disconnected LCA groups}, Topology Appl. 169 (2014): 21--32.

\bibitem{DGBpet} D. Dikranjan, A. Giordano Bruno, \emph{Entropy on abelian groups}, Adv. Math. 298 (2016) 612--653.

\bibitem{DGBS} D. Dikranjan, A. Giordano Bruno, L. Salce, \emph{Adjoint algebraic entropy}, J. Algebra 324 (2010): 442--463.

\bibitem{intrinsic} D. Dikranjan, A. Giordano Bruno, L. Salce, S. Virili, \emph{Intrinsic algebraic entropy}, J. Pure Appl. Algebra 219 (2015): 2933--2961.

\bibitem{Aeag} D. Dikranjan, B. Goldsmith, L. Salce, P. Zanardo, \emph{Algebraic entropy for abelian groups}, Trans. Amer. Math. Soc. 361 (2009): 3401--3434.

\bibitem{Dik+Manolo} D. Dikranjan, M. Sanchis, \emph{Bowen's entropy for endomorphisms of totally bounded abelian Groups}, Descriptive Topology and Functional Analysis, Springer Proceedings in Mathematics \& Statistics, Volume 80 (2014): 143--162.

\bibitem{DSV} D. Dikranjan, M. Sanchis, S. Virili, \emph{New and old facts about entropy on uniform spaces and topological groups}, Topology Appl. 159 (2012): 1916--1942.

\bibitem{Gabriel} P. Gabriel, \emph{Des cat\'egories ab\'eliennes}, Bull. Soc. Math. France 90 (1962).

\bibitem{GB*} A. Giordano Bruno, \emph{Adjoint entropy vs Topological entropy}, Topology Appl. 159 (2012): 2404--2419.

\bibitem{GB} A. Giordano Bruno, \emph{Topological entropy for automorphisms of totally disconnected locally compact groups}, Topology Proc. 45 (2015): 175--187

\bibitem{GBSalce} A. Giordano Bruno, L. Salce, \emph{A soft introduction to algebraic entropy}, Arab. J. Math. 1 (2012): 69--87.

\bibitem{yuzapp} A. Giordano Bruno, S. Virili, \emph{On the Algebraic Yuzvinski Formula}, Topol. Algebra and its Appl. 3 (2015): 86--103.

\bibitem{GBVirili} A. Giordano Bruno, S. Virili, \emph{Topological entropy in totally disconnected locally compact groups}, Ergodic Theory and Dynamical Systems (2016): to appear, doi:10.1017/etds.2015.139.

\bibitem{Gr} A. Grothendieck, \emph{Sur quelques points d'alg\'ebre homologique}, Tohoku Mathematical Journal, Second Series 9.2 (1957): 119--183.

\bibitem{H} B. M. Hood,  \emph{Topological entropy and uniform spaces}, J. London Math. Soc. 8 (1974): 633--641. 

\bibitem{Mat} E. Matlis, \emph{Injective modules over Noetherian rings}, Pacific J. Math. 8 (1958): 511--528.

\bibitem{Kothe} G. K{\"o}the, \emph{Topological vector spaces}, Springer Berlin Heidelberg, 1983.

\bibitem{Lef} S. Lefschetz, \emph{Algebraic topology}, Vol. 27. American Mathematical Soc., 1942.

\bibitem{Peters1} J. Peters,  \emph{Entropy on discrete abelian groups},  Adv. Math. 33 (1979): 1--13.

\bibitem{Peters2} J. Peters, \emph{Entropy of automorphisms on LCA groups}, Pacific J. Math. 96 (1981): 475--488.

\bibitem{profinite} L. Ribes, and P. Zalesskii. {\it Profinite groups.} Profinite Groups. Springer Berlin Heidelberg, (2000). 19-77.

\bibitem{SZ} L. Salce, P. Zanardo, \emph{A general notion of algebraic entropy and the rank-entropy},  Forum Math. 21 (2009): 579--599.

\bibitem{SVV} L. Salce,  P. V\'amos, S. Virili, \emph{Length functions, multiplicities and algebraic entropy}, Forum Math. 25 (2013): 255--282.

\bibitem{SV} L. Salce, S. Virili, \emph{The Addition Theorem for algebraic entropies induced by non-discrete length functions}, Forum Math. 28 (2016): 1143--1157.

\bibitem{St} L. N. Stoyanov, \emph{Uniqueness of topological entropy for endomorphisms on compact groups}, Boll. Un. Mat. Ital. B (7) 1 (1987): 829--847.

\bibitem{vD} D. van Dantzig, \emph{Studien over topologische Algebra}, Dissertation, Amsterdam 1931.

\bibitem{Virili-BT} S. Virili, \emph{Algebraic and topological entropy of group actions}, preprint.

\bibitem{Virili} S. Virili, \emph{Entropy for endomorphisms of LCA groups}, Topology Appl. 159 (2012): 2546--2556.

\bibitem{Warner} S. Warner, {\it Topological rings}. Vol. 178. Elsevier, 1993.

\bibitem{Weiss} M. D. Weiss,  \emph{Algebraic and other entropies of group endomorphisms}, Theory of Computing Systems 8 (1974): 243--248.

\bibitem{Willis} G. A. Willis, \emph{The scale and tidy subgroups for endomorphisms of totally disconnected locally compact groups}, Math. Ann. 361 (2015): 403--442.

\bibitem{Y} S. Yuzvinski, \emph{Metric properties of endomorphisms of compact groups}, Izv. Acad. Nauk SSSR, Ser. Mat. 29 (1965): 1295--1328 (in Russian). English Translation: Amer. Math. Soc. Transl. 66 (1968): 63--98.

\end{thebibliography}
\end{document}